\documentclass[11pt,letterpaper]{amsart}

\usepackage[T1]{fontenc}
\usepackage{lmodern}
\usepackage[expansion=false]{microtype}
\usepackage{mathtools,amssymb,amsfonts,amsthm}
\usepackage[margin=1.08in]{geometry}
\usepackage{enumitem}
\usepackage{booktabs,tabularx,array}
\usepackage{xcolor}
\usepackage{hyperref}
\usepackage[alphabetic]{amsrefs}
\usepackage[capitalize,noabbrev]{cleveref}

\makeatletter
\AtBeginDocument{%
  \def\cref@getref#1#2{%
    \protected@edef#2{\getrefbykeydefault{#1@cref}{}{}}}%
  \def\cpageref@getref#1#2{%
    \protected@edef#2{\getrefbykeydefault{#1@cref}{page}{}}}%
}
\makeatother

\numberwithin{equation}{section}
\setlist[itemize]{leftmargin=2em,itemsep=2pt,topsep=4pt}
\setlist[enumerate]{leftmargin=2.5em,itemsep=3pt,topsep=4pt}
\hypersetup{
  colorlinks=true,
  linkcolor=blue,
  citecolor=red,
  urlcolor=blue,
  pdftitle={Failure of the local decay formula for Arnold multiplicities under twisted Kahler-Ricci flow},
  pdfauthor={Xiangsen Qin},
  pdfsubject={Arnold multiplicities of maximal weak solutions of the twisted Kahler-Ricci flow},
  pdfkeywords={twisted Kahler-Ricci flow, Arnold multiplicity, divisorial Zariski decomposition, Hirzebruch surface}
}

\theoremstyle{plain}
\newtheorem{theorem}{Theorem}[section]
\newtheorem{proposition}[theorem]{Proposition}
\newtheorem{lemma}[theorem]{Lemma}
\newtheorem{corollary}[theorem]{Corollary}
\newtheorem{dglquestion}{Question}

\crefname{dglquestion}{Question}{Questions}
\Crefname{dglquestion}{Question}{Questions}

\newtheorem{remark}[theorem]{Remark}

\newcommand{\CC}{\mathbb C}
\newcommand{\PP}{\mathbb P}
\newcommand{\PSH}{\operatorname{PSH}}
\newcommand{\Supp}{\operatorname{Supp}}
\newcommand{\vol}{\operatorname{vol}}
\newcommand{\pr}{\operatorname{pr}}
\newcommand{\ddc}{dd^c}
\newcommand{\lct}{c}
\newcommand{\arn}{\lambda}
\newcommand{\J}{\mathcal J}
\newcommand{\E}{\mathcal E}
\begin{document}

\title[Failure of the local decay formula]
{Failure of the local decay formula for Arnold multiplicities
under twisted K\"ahler--Ricci flow}

\author[X. Qin]{Xiangsen Qin}

\address{Xiangsen Qin: Chern Institute of Mathematics and LPMC, Nankai University \\
Tianjin 300071, China}

\email{qinxiangsen@nankai.edu.cn}
\date{}

\begin{abstract}
Di Nezza--Guedj--Lu asked whether local Arnold multiplicities under
maximal weak twisted K\"ahler--Ricci flow with fixed cohomology class satisfy
$\lambda(\varphi_t,x)=\max\{\lambda(\varphi_0,x)-t,0\}$.
We construct counterexamples with analytic initial singularities in every
dimension $n\ge2$.
On the Hirzebruch surface $\mathbb F_e$, $e\ge2$, the initial current
$[S]+(e+1)[F]$ yields $\lambda(\varphi_t,x)=1/2-t/e$ for
$x\in S\setminus F$ and $0<t<e/2$, where $S$ is the negative section
and $F$ a fibre.  Products with smooth stationary factors give
higher-dimensional examples.
The Zariski negative part of $\{S\}+\{F\}$, the reduced support class,
forces extra divisor mass.
For pure divisorial data with simple normal crossings support, a smooth
semipositive representative of the reduced support class's positive part
gives exact short-time formulas.
On one surface we construct flows with identical initial current germs,
background data and global initial Arnold multiplicity, but different
local rates.
\end{abstract}

\subjclass[2020]{32U05, 32U20, 32W20, 14C20, 53E20.}
\keywords{Twisted K\"ahler--Ricci flow, Arnold multiplicity,
divisorial Zariski decomposition, Hirzebruch surface.}

\maketitle

\tableofcontents
\clearpage

\section{Introduction}
\label{sec:introduction}

Weak K\"ahler--Ricci flow provides a parabolic regularization of closed
positive currents, allowing singular initial data beyond smooth
K\"ahler metrics.  The quantitative problem is to determine how the
singularities change during the interval in which they persist.
For a quasi-plurisubharmonic function $u$, its
\emph{local Arnold multiplicity} at $x$ is defined by
\[
 \arn(u,x)=c(u,x)^{-1},\qquad
 c(u,x)=\sup\{q>0:e^{-qu}\in L^1_{\rm loc}\text{ near }x\},
\]
with the convention $1/(+\infty)=0$.
Let $(X,\omega)$ be a compact K\"ahler manifold and $dV_X$ a smooth
positive volume form.  Starting from $\varphi_0\in\PSH(X,\omega)$,
the maximal weak flow is obtained by decreasing smooth approximation;
on its regular open sets it satisfies
\begin{equation}\label{eq:flow-intro}
 (\omega+\ddc\varphi_t)^n=e^{\dot\varphi_t}dV_X,
 \qquad n=\dim_{\CC}X,\qquad
 \ddc=\frac{i}{\pi}\partial\bar\partial.
\end{equation}
In terms of $\omega_t=\omega+\ddc\varphi_t$, this is the fixed-class
twisted K\"ahler--Ricci flow
\begin{equation}\label{eq:twisted-flow-intro}
 \partial_t\omega_t=-\operatorname{Ric}(\omega_t)+\chi,
 \qquad \chi=\operatorname{Ric}(dV_X),
\end{equation}
on its regular locus.  The reference volume fixes the twisting form
$\chi$ and need not equal $\omega^n$; the construction and normalization
are recalled in \cref{sec:conventions}.

Zero-Lelong initial currents become smooth immediately \cite{GZ17flow},
whereas more singular data can retain positive Lelong numbers while
regularizing away from analytic subsets \cite{DNL17}.
The local Arnold multiplicity measures the remaining singularity and is
unchanged by bounded perturbations; its decrease over time measures the decay of singularities.  

Write $\arn(u):=\sup_{x\in X}\arn(u,x)$
for the global Arnold multiplicity. Di Nezza--Guedj--Lu proved, when $0<\arn(\varphi_0)<+\infty$, that
\begin{equation}\label{eq:DGL-two-sided}
 \max\{\arn(\varphi_0,x)-t,0\}
 \le\arn(\varphi_t,x)
 \le\max\!\left\{\arn(\varphi_0,x)
 \left(1-\frac{t}{\arn(\varphi_0)}\right),0\right\};
\end{equation}
see \cite{DGL26geometry}*{Theorem~1.9}.
The lower estimate says that the local multiplicity cannot decrease
faster than unit speed.  Equality holds on compact Riemann surfaces
\cite{DGL26geometry}*{Lemma~2.1}.  The upper estimate also involves the
largest initial multiplicity anywhere on $X$, and can leave a strict gap.
These bounds determine the global quantity exactly:
$\arn(\varphi_t)=\max\{\arn(\varphi_0)-t,0\}$.
This identity controls the largest remaining Arnold multiplicity but
does not determine its values at individual points.

Di Nezza--Guedj--Lu asked whether the same formula holds pointwise
\cite{DGL26geometry}*{Question~1.10}.
\begin{dglquestion}\label[dglquestion]{q:DGL-decay-formula}
Does every maximal weak solution $(\varphi_t)$ satisfy
\begin{equation}\label{eq:unit-slope-question}
 \arn(\varphi_t,x)=\max\{\arn(\varphi_0,x)-t,0\}
\end{equation}
for every point $x$ and every time $t$ for which the flow is defined?
\end{dglquestion}

We give a negative answer in every complex dimension $n\ge2$, with
analytic initial singularities.  The basic example is a divisorial
current on a Hirzebruch surface, where the negative part of the class
of the reduced initial support forces additional divisor mass.
Products with smooth stationary factors give the higher-dimensional
examples.  The known global decay remains valid in all these examples;
the local multiplicities retain information about the global divisor
configuration.

\subsection*{Main results}
For an integer $e\ge2$, the obstruction appears on the Hirzebruch surface
\[
 \pi_e:\mathbb F_e=\PP_{\PP^1}
 (\mathcal O_{\PP^1}\oplus\mathcal O_{\PP^1}(-e))
 \longrightarrow\PP^1,
\]
where we use the quotient convention for projective bundles.
Let $S$ be the section induced by the quotient onto
$\mathcal O_{\PP^1}(-e)$, and let $F$ be a fibre.
Then $S^2=-e$, $S\cdot F=1$ and $F^2=0$.
For a divisor $D$, write $[D]$ for its integration current and $\{D\}$
for its cohomology class; denote the fibre class by $\mathfrak f$. The following theorem gives a counterexample to \cref{q:DGL-decay-formula} on $\mathbb F_e$.

\begin{theorem}
\label[theorem]{cor:exact-counterexample}
Fix an integer $e\ge2$ and a fibre $F_1$ of $\mathbb F_e$.
Choose a K\"ahler form $\omega$ in the class
$\{S\}+(e+1)\mathfrak f$, any smooth positive volume form $dV_X$, and
$\varphi_0\in\PSH(\mathbb F_e,\omega)$ with
\[
 \omega+\ddc\varphi_0=[S]+(e+1)[F_1].
\]
The initial potential has analytic singularities.  Its maximal weak flow
satisfies, for $x\in S\setminus F_1$ and $0<t<e/2$,
\begin{equation}\label{eq:Q110-counterexample}
 \arn(\varphi_t,x)=\frac12-\frac te
 >\max\!\left\{\frac12-t,0\right\}
 =\max\{\arn(\varphi_0,x)-t,0\}.
\end{equation}
\end{theorem}

The product construction in \cref{thm:product} extends this counterexample
to every complex dimension $n\ge2$.
For $n>2$, take $X_n=\mathbb F_e\times\PP^{n-2}$ and the stationary
flow with zero potential on the second factor.  With the normalized product
reference volume of \cref{thm:product}, the maximal weak solution is
$\Phi_t=\operatorname{pr}_1^*\varphi_t$.  At $(x,y)$ with
$x\in S\setminus F_1$, its Arnold multiplicity is the same
$1/2-t/e$ throughout $0<t<e/2$; see \cref{cor:all-dimensions}.
The product initial current includes a smooth term from the second factor.
This transfer uses the surface example and the product theorem,
independently of the higher-dimensional exactness theorem below.

To explain this slower local decay, we consider the positivity of the
evolving current in its fixed cohomology class.  Divisorial persistence
and the Zariski negative part of the reduced support class
$\{S\}+\mathfrak f$ force additional divisor mass along $S$, yielding
$\arn(\varphi_t,x)\ge1/2-t/e$.  This lower bound already gives a negative
answer to \cref{q:DGL-decay-formula}.  A short proof for arbitrary
positive weights $a,b$ satisfying $b>ea$ appears in
\cref{thm:robust-counterexample} in \cref{sec:short-counterexample}.

The following theorem extends this computation to pure divisorial
initial data in any dimension, under the geometric assumptions
stated below.
Write $\beta=P+\{N\}$ for the divisorial Zariski decomposition, with
$N$ the effective negative part.  In higher dimensions $P$ is generally
only modified nef; the smooth semipositive representative assumed below
makes it nef.  A divisor has simple normal crossings (SNC) if its
components are locally coordinate hyperplanes in suitable holomorphic
coordinates.
\begin{theorem}
\label[theorem]{thm:surface-exactness}
Let $(X,\omega)$ be a compact K\"ahler manifold of complex dimension
$n\ge1$, fix a smooth positive volume form $dV_X$, and let
$D=D_1+\cdots+D_r$, $r\ge1$, be a reduced simple normal crossings divisor.
Suppose $\varphi_0\in\PSH(X,\omega)$ satisfies
\begin{equation}\label{eq:pure-initial}
 \omega+\ddc\varphi_0=\sum_{i=1}^r m_i[D_i],\qquad m_i>0.
\end{equation}
The identity \eqref{eq:pure-initial} makes $\beta:=\sum_i\{D_i\}$ big,
and its divisorial Zariski decomposition has the form
$\beta=P+\{N\}$ with $N=\sum_i\eta_iD_i$ and $0\le\eta_i<1$.
Assume that $P$ has a smooth semipositive representative.
Then the maximal weak solution $(\varphi_t)$ from $\varphi_0$ satisfies
\begin{equation}\label{eq:surface-exact-profile}
 \arn(\varphi_t,x)=
 \max\Bigl(\{0\}\cup
 \{m_i/2-t(1-\eta_i):1\le i\le r,\ x\in D_i\}\Bigr)
\end{equation}
for every $x\in X$ and
\[
 0<t<\tau_Z:=\min_i\frac{m_i}{2(1-\eta_i)}.
\]
\end{theorem}

The correction $t\eta_i$ in \eqref{eq:surface-exact-profile}
is the term forced by that negative part.
The hypotheses hold for the Hirzebruch families of
\cref{thm:Hirz-exact}.  The initial class identity also forces $X$
to be projective, as explained in \cref{sec:surface-exactness}.

The next theorem isolates the nonlocal effect on surfaces: the initial
current germ,
the background data, and even the global initial Arnold multiplicity can
all be held fixed while the local rates vary.

\begin{theorem}
\label[theorem]{thm:parametric-nonlocality}
For every integer $M\ge2$ and every real $a>0$, there exist an
integer $e\ge M+2$, a K\"ahler form $\omega$ and a smooth positive volume
form $dV_X$ on $X=\mathbb F_e$, a point $x\in S$, and potentials
$\varphi_0^{(1)},\ldots,\varphi_0^{(M)}\in\PSH(X,\omega)$ with analytic
singularities such that
\begin{equation}\label{eq:intro-common-germ}
 \omega+\ddc\varphi_0^{(j)}=a[S]\quad\text{near }x
 \qquad(1\le j\le M),
\end{equation}
and
\begin{equation}\label{eq:intro-common-global-arnold}
 \arn(\varphi_0^{(1)})=\cdots=\arn(\varphi_0^{(M)}).
\end{equation}
Let $(\varphi_t^{(j)})$ be the maximal weak solution with initial potential
$\varphi_0^{(j)}$.  Then
\begin{equation}\label{eq:intro-nonlocal-slopes}
 \arn(\varphi_t^{(j)},x)=\frac a2-\frac{j+1}{e}t,
 \qquad 1\le j\le M,\quad 0<t<a/2.
\end{equation}
\end{theorem}

The construction varies the divisor configuration away from the chosen
point.  Thus neither the initial current germ nor that germ together with
the global initial index determines the local evolution.  Equality of
current germs fixes the local singularity type, not the chosen potential
functions.  The proof is given in \cref{sec:applications}.

\subsection*{Relation to prior work and proof strategy}
For purely divisorial initial data,
\cite{DGL26geometry}*{Proposition~3.6} gives an exact potential formula
when the cohomology class governing the residual current after
divisorial extraction along the flow is nef and big.
The discussion after Question~3.5 explains how a Zariski decomposition
with nef positive part can provide regularity on a dense open set.
Lemma~4.7 computes the negative part in an example, and Remark~4.8
anticipates the corresponding additional divisorial terms in the flow.
The results here give explicit strict counterexamples,
exact point Lelong numbers, Arnold multiplicities, and multiplier ideals
when the reduced support class $\beta$ may be non-nef, and different
local decay rates on a fixed surface with the initial current germ, background data,
and global initial Arnold multiplicity fixed.
These formulas determine local invariants; the semipositivity
assumption makes $P$ and $2tP$ nef.

In the ample SNC setting, \cite{DGL26cusps}*{Proposition~2.5} gives a
potential formula under special background identities, while
Proposition~2.7 determines the singularity type on the initial short-time
interval with an arbitrary smooth positive reference volume and a
semipositive smooth residual background.  Its ample-divisor hypotheses
exclude the negative section and fibre in our surface example, whose
reduced support class $\beta=\{S\}+\mathfrak f$ is non-nef for $e\ge2$.

For the strict counterexample, divisorial persistence makes the current
left by the signed extraction positive.  The Zariski negative part of
its class then forces additional divisor mass and gives the strict
lower bound.

For exactness, \cite{DGL26geometry}*{Theorem~3.4} provides a weighted
elliptic solution in a general big class.  After removing the fixed
part, we obtain the regularity needed for parabolic comparison from
\cite{BG14}*{Theorem~4.6}.  We use this solution to construct a
parabolic comparison candidate.  The resulting potential comparison,
together with \cite{BEGZ10}, gives full mass in $2tP$ for the actual
remainder after removing the corrected divisor terms.
The pointwise theorem of \cite{DDL18} gives zero
Lelong numbers for this remainder at every point.  Local
Skoda--H\"older integrability then determines the Arnold multiplicities
and multiplier ideals.

The paper proceeds from conventions in \cref{sec:conventions} to the
short counterexample in \cref{sec:short-counterexample}, then to the
exact divisorial formula and its Hirzebruch examples in
\cref{sec:surface-exactness,sec:Hirzebruch}.
Nonlocality and the applications follow in
\mbox{Sections}~\ref{sec:applications} and~\ref{sec:products}.  Appendix~\ref{app:resolution-tools} retains the
resolution bounds for DGL-compatible analytic initial potentials,
which allow non-pure-divisorial and non-SNC singularities but impose
an additional coefficient restriction.  It also records the classical
surface chamber calculation; neither is needed for the short counterexample.

\section{Preliminaries}
\label{sec:conventions}

\subsection{Notation and conventions}

We work on a connected component throughout; on a disconnected compact
manifold the arguments apply componentwise and additive constants may
be chosen separately on each component.

Our differential-operator normalization is
\begin{equation}\label{eq:ddc-normalization}
 d^c=\frac{i}{2\pi}(\bar\partial-\partial),
 \qquad \ddc=\frac{i}{\pi}\partial\bar\partial.
\end{equation}
We use the exponential convention
\begin{equation}\label{eq:multiplier-definition}
 \J(q u)_x=\{f\in\mathcal O_{X,x}:|f|^2e^{-qu}\in L^1_{\rm loc}\},
 \qquad q>0.
\end{equation}
For a single complex coordinate $z_1$, the three local calibration
identities are
\begin{equation}\label{eq:local-calibration}
 \ddc\log|z_1|=[z_1=0],\qquad
 \arn(a\log|z_1|)=a/2,\qquad
 \J(qa\log|z_1|)=(z_1^{\lfloor qa/2\rfloor})\quad(a\ge0).
\end{equation}
Potential-level flow estimates are invariant under a positive scaling of
$\ddc$.  If $\widetilde{dd^c}=s\ddc$, take
$\widetilde\omega=s\omega$ and $d\widetilde V=s^n dV_X$; then
\[
 (\widetilde\omega+\widetilde{dd^c}\varphi_t)^n
 =s^n(\omega+\ddc\varphi_t)^n=e^{\dot\varphi_t}d\widetilde V.
\]
The potentials, time variable, and decreasing smooth approximants are
unchanged, so smooth uniqueness identifies their maximal weak limits.
Backgrounds and twisting forms scale by $s$ as well.  In importing
\cite{DGL26geometry}*{Theorem~1.9}, we use its potential-level statement
for the index defined by $e^{-qu}$; divisor coefficients below are always
measured against canonical integration currents in
\eqref{eq:ddc-normalization}.  The resolution barrier is computed in
this convention in Appendix~\ref{app:resolution-tools}.

For a closed real $(1,1)$-form or current $T$ on $X$, the notation
$\{T\}$ means its real $(1,1)$-cohomology class in
$H^{1,1}(X,\mathbb R)$.
For an effective divisor $D$, let $s_D$ be the canonical defining
section of $\mathcal O_X(D)$, and choose a smooth Hermitian metric $h_D$
on this line bundle.  Write $\ell_D=\log|s_D|_{h_D}$, and let $\theta_D$ be
the smooth real curvature $(1,1)$-form representing $\{D\}$,
characterized by $\theta_D=-\ddc\log|s_D|_{h_D}$ off $D$.  Thus
$\theta_D$ is the curvature of the logarithmic norm, not of its square.
The Poincar\'e--Lelong identity is
\begin{equation}\label{eq:PL}
 \ddc\ell_D=[D]-\theta_D.
\end{equation}
We distinguish $D$, its integration current $[D]$, and its class $\{D\}$;
in set-theoretic expressions $D$ denotes its support.
If $E=\sum_j a_jE_j$ is an effective real divisor with distinct prime
components $E_j$, its \emph{associated reduced divisor} is
$E_{\rm red}:=\sum_jE_j$.  A reduced divisor has \emph{simple normal
crossings} (SNC) if its irreducible components are smooth and, in
suitable local holomorphic coordinates, are given by coordinate hyperplanes.

For a smooth closed real $(1,1)$-form $\theta$, we use
$\PSH(X,\theta)=\{u\text{ quasi-psh}:\theta+\ddc u\ge0\}$,
with $u\not\equiv-\infty$.  Fix a K\"ahler form $\omega_X$ on $X$.
Every quasi-psh function on compact $X$ belongs to
$\PSH(X,C\omega_X)$ for some $C>0$.

A quasi-psh function has \emph{neat analytic singularities} if, on every sufficiently
small coordinate neighbourhood, there are an integer $N\ge1$, holomorphic
functions $f_1,\ldots,f_N$, a number $c_0>0$, and a smooth function $h$
such that
\begin{equation}\label{eq:analytic-singularities}
 u=\frac{c_0}{2}\log\!\left(\sum_{\ell=1}^N|f_\ell|^2\right)+h.
\end{equation}
We abbreviate neat analytic singularities to analytic singularities.
Two quasi-plurisubharmonic functions $u$ and $v$ have the same
\emph{singularity type} near a point if, on some neighbourhood of that
point, $u\le v+C$ and $v\le u+C$ for a constant $C>0$.  Two
germs of closed positive $(1,1)$-currents have the same local singularity
type if their local potentials do; this condition is independent of the
chosen smooth local representatives.
Given an integer $r\ge1$, local defining functions
$s_1,\ldots,s_r$ for prime divisors, and real weights
$m_1,\ldots,m_r>0$, we call an expression
\[
 \sum_{i=1}^rm_i\log|s_i|+O(1),\qquad m_i>0,
\]
with arbitrary positive real weights a \emph{real divisorial logarithmic
singularity}, where $O(1)$ denotes a locally bounded term.  If this term is
smooth and $m_i=c_*N_i$ for some $c_*>0$ and integers $N_i\ge1$, the
singularity is analytic in the sense of \eqref{eq:analytic-singularities}.

\subsection{Lelong numbers and Arnold multiplicities}
\label{subsec:arnold-invariants}

Let $x\in X$ and let $u$ be quasi-pluri\-subharmonic near $x$.
Choose holomorphic coordinates $z$ centred at $x$ and a smooth local
function $g$ such that $u+g$ is plurisubharmonic.  We define the point
Lelong number by
\begin{equation}\label{eq:point-lelong-definition}
 \nu_x(u):=\lim_{r\downarrow0}
 \frac{\sup_{\|z\|\le r}(u(z)+g(z))}{\log r}.
\end{equation}
The limit is independent of $g$ and of the chosen coordinates.  For a
closed positive $(1,1)$-current $T$, $\nu_x(T)$ denotes the Lelong number
of any local plurisubharmonic potential of $T$.  This convention gives
$\nu_0(\log|z|)=1$; see
\cite{Dem-book}*{Chapter~III, \S\S5--7}.
For a prime divisor $G$, its generic Lelong coefficient is
$\nu_G(T):=\inf_{x\in G}\nu_x(T)$; this is the coefficient of $[G]$
in Siu's divisorial decomposition of $T$.
The local complex singularity exponent and Arnold multiplicity are
\begin{equation}\label{eq:arnold-def}
 \lct(u,x)=\sup\{q>0:\exp(-qu)\in L^1_{\mathrm{loc}}\text{ near }x\},
 \qquad \arn(u,x)=\lct(u,x)^{-1}.
\end{equation}
We use the reciprocal conventions $1/(+\infty):=0$ and
$1/0:=+\infty$.  The global invariant is $\arn(u)=\sup_X\arn(u,x)$.
The same integral definitions also apply to measurable potentials,
without requiring them to be quasi-psh.
For quasi-psh $u$, Skoda's exponential-integrability theorem
\cite{GZ17book}*{Theorem~2.50}, with \eqref{eq:ddc-normalization}, implies
that $\nu_x(u)=0$ gives $\exp(-A u)\in L^1_{\mathrm{loc}}$ for every finite
$A>0$.

The next lemma computes Arnold multiplicities and multiplier ideals
for SNC logarithmic poles with a remainder of zero point Lelong number.

\begin{lemma}[SNC integrability]\label[lemma]{lem:SNC}
Let $D_j=\{z_j=0\}$, $1\le j\le r\le n$, near $0\in\CC^n$, allowing
$r=0$.  Suppose $c_j\ge0$, $u$ is quasi-psh with $\nu_0(u)=0$, and $v$
is locally bounded and measurable.  Set
$\phi=\sum_{j=1}^r c_j\log|z_j|+u+v$ and
$C=\max(\{0\}\cup\{c_j:1\le j\le r\})$.  Then
\begin{equation}\label{eq:SNC-arnold}
 \arn(\phi,0)=C/2,\qquad
 \J(q\phi)_0=\left(\prod_{j=1}^r z_j^{\lfloor qc_j/2\rfloor}\right)
 \quad(q>0).
\end{equation}
If $T$ is a closed positive $(1,1)$-current and $T\ge c[D]$ for $c\ge0$,
then its local potentials have Arnold multiplicity at least $c/2$ at
every smooth point of $D$.
\end{lemma}

\begin{proof}
The bounded term $v$ does not affect integrability.  Since $u$ is locally
bounded above, a transverse one-variable integral shows that
$e^{-q\phi}$ is not integrable if $qC\ge2$ and $C>0$.
If $qC<2$, choose $p>1$ with $pqC<2$ and let $p'=p/(p-1)$.
The monomial weight $\prod_j|z_j|^{-qc_j}$ is in $L^p$, while Skoda's
theorem gives $e^{-qp'u}\in L^1_{\rm loc}$.  H\"older's inequality proves
integrability.  If $C=0$, Skoda applies for every finite $q$.

For the ideal formula, fix $q>0$ and a nonzero holomorphic germ $f$.
Write $r_j=\operatorname{ord}_{D_j}f$.  At a nearby point of $D_j$ off the
other components where $f/z_j^{r_j}$ is nonzero, the upper bound for $u$
shows that integrability requires
$2r_j-qc_j>-2$, equivalently $r_j\ge\lfloor qc_j/2\rfloor$.
This argument applies inside any neighbourhood of the original point,
so proves necessity at its stalk, including crossings.
Conversely, factor $f=(\prod_j z_j^{\lfloor qc_j/2\rfloor})g$ with $g$
holomorphic.  Put
$\delta_j=qc_j/2-\lfloor qc_j/2\rfloor\in[0,1)$ and choose $p>1$ with
$p\delta_j<1$ for every $j$.  After cancellation, the divisorial weight
is $\prod_j|z_j|^{-2\delta_j}\in L^p$; Skoda and H\"older again give
integrability of $|f|^2e^{-q\phi}$.  If $qc_j/2$ is an integer, then
$\delta_j=0$, whereas one fewer vanishing order gives exponent $-2$ and
divergence.  Thus the formula includes the integer endpoints exactly.

For the last assertion, a local potential of $T$ can be written
$c\log|f_D|+w+h$, where $w$ is psh and $h$ is smooth.  Their local upper
bound gives non-integrability when $qc\ge2$ by the same transverse
integral; the case $c=0$ is immediate.
\end{proof}

\subsection{Positive classes and divisorial Zariski decompositions}
\label{subsec:zariski-preliminaries}

In this subsection $X$ is a compact K\"ahler manifold of complex dimension
$n$, and $\omega_X$ is a fixed K\"ahler form on $X$.  A real
$(1,1)$-class $\alpha$ is
\emph{pseudoeffective} if it contains a
closed positive current; it is \emph{nef} if for every $\varepsilon>0$ it
has a smooth representative bounded below by $-\varepsilon\omega_X$; it
is \emph{big} if it contains a K\"ahler current, namely a closed current
$T\ge\varepsilon\omega_X$ for some $\varepsilon>0$; and it is K\"ahler if
it contains a K\"ahler form.  A positive current $T_{\min}\in\alpha$ has
\emph{minimal singularities} if its potential is less singular than the
potential of every other positive current in $\alpha$.
For potentials relative to the same smooth representative, $u$ is less
singular than $v$ if $u\ge v-C$ for some constant $C$.

For a pseudoeffective class $\alpha$, Boucksom's generic minimal
multiplicity along a prime divisor $G$ is written $\nu_G(\alpha)$.
It is the limit, as $\varepsilon\downarrow0$, of the infimum of generic
Lelong numbers along $G$ over closed positive currents in the perturbed
class $\alpha+\varepsilon\{\omega_X\}$.  Its divisorial Zariski
decomposition is
\begin{equation}\label{eq:DZD}
 \alpha=P(\alpha)+\{N(\alpha)\},
 \qquad N(\alpha)=\sum_G\nu_G(\alpha)G.
\end{equation}
Thus $N(\alpha)$ is an effective real divisor and $[N(\alpha)]$ is its
positive integration current.  The definition and the Siu decomposition
give
\begin{equation}\label{eq:Bou-domination}
 T\ge[N(\alpha)]\qquad\text{for every closed positive current $T$ with }
 \{T\}=\alpha.
\end{equation}
On a compact K\"ahler surface, $P(\alpha)$ is nef, its intersection with
every component of $N(\alpha)$ is zero, and the intersection matrix of
those components is negative definite \cite{Bou04}.

\subsection{Maximal weak solutions and twisted K\"ahler--Ricci flow}
\label{subsec:maximal-flow}

Fix a smooth positive reference volume form $dV_X$.  In a holomorphic
coordinate chart, write $dV_X=\sigma_V\,dV_{\rm eucl}$, where
$dV_{\rm eucl}$ is the Euclidean volume form and $\sigma_V$ is a smooth
positive density, and define its
\emph{Ricci form} by
\begin{equation}\label{eq:ricci-reference-volume}
 \operatorname{Ric}(dV_X):=-\ddc\log\sigma_V.
\end{equation}
Under a holomorphic coordinate change, the two densities differ by the
modulus squared of a nonvanishing holomorphic Jacobian.  Its logarithm is
pluriharmonic, so \eqref{eq:ricci-reference-volume} defines a global smooth
closed real $(1,1)$-form.  This fixed form is the \emph{twisting form}.

For $\varphi_0\in\PSH(X,\omega)$ and $j\in\mathbb N$, use \cite{BK07}*{Theorem~1} to choose smooth $\omega$-psh decreasing
approximants.  The strict-positivity adjustment below then gives smooth
strictly $\omega$-psh functions $\varphi_{0,j}\downarrow\varphi_0$.  Let
$\varphi_{t,j}$ be the smooth solutions of \eqref{eq:flow-intro}.  The
pointwise decreasing limit is independent of the approximation and is the
\emph{maximal weak solution} of the parabolic complex Monge--Amp\`ere
equation.  The associated current path
$\omega_t:=\omega+\ddc\varphi_t$ is the weak twisted
K\"ahler--Ricci flow determined by $(\varphi_t)$, and it satisfies
\eqref{eq:twisted-flow-intro} on its regular locus.  As $t\downarrow0$,
the potentials converge to $\varphi_0$ in $L^1$ and
in capacity and are smooth on each open set furnished by the regularity results
\cites{GZ17flow,DNL17}; see also \cite{DGL26cusps}*{Theorems~1.11 and~1.13}.  All comparisons below refer to this specific
monotone construction.  Strict positivity can be imposed without changing
the limit: if $u_j\downarrow\varphi_0$ is a smooth $\omega$-psh
approximation, choose $C\ge\sup_Xu_1$ and $\varepsilon_j\downarrow0$
with $0<\varepsilon_j<1$.  Then
$v_j=(1-\varepsilon_j)(u_j-C)+C$ decreases to $\varphi_0$ and
$\omega+\ddc v_j\ge\varepsilon_j\omega>0$.
Equalities of currents are distributional on $X$;
parabolic equations are classical on the explicitly stated open set.

\section{A short counterexample}
\label{sec:short-counterexample}

The strict counterexample needs only divisorial persistence and domination
by a Zariski negative part.  No elliptic construction or matching upper
bound enters this section.

\begin{theorem}
\label[theorem]{thm:robust-counterexample}
Fix an integer $e\ge2$.  Let $a,b>0$ satisfy $b>ea$, and fix a fibre $F_0$ of
$\pi_e$.  Fix a
K\"ahler form $\omega$ on
$\mathbb F_e$ with $\{\omega\}=a\{S\}+b\mathfrak f$, fix a smooth
positive volume form $dV_X$, and choose
$\varphi_0\in\PSH(\mathbb F_e,\omega)$ such that
\begin{equation}\label{eq:robust-initial}
 \omega+\ddc\varphi_0=a[S]+b[F_0].
\end{equation}
Then the maximal weak solution $(\varphi_t)$ satisfies
\begin{equation}\label{eq:robust-lower-bound}
 \arn(\varphi_t,x)\ge \frac a2-\frac{t}{e}
 >\max\{a/2-t,0\}
 =\max\{\arn(\varphi_0,x)-t,0\}
\end{equation}
for every $x\in S\setminus F_0$ and $0<t<ea/2$.
\end{theorem}

The following persistence statement allows arbitrary positive real
divisor coefficients on the original manifold and does not require
neat analytic singularities.  Its proof uses a generic-point argument.

\begin{lemma}\label[lemma]{lem:real-persistence}
Let $(X,\omega)$ be a compact K\"ahler manifold, let $dV_X$ be a smooth
positive volume form, and let $\varphi_0\in\PSH(X,\omega)$.  Let
$D_1,\ldots,D_r\subset X$, with $r\ge0$, be distinct prime divisors and suppose
\begin{equation}\label{eq:real-initial}
 \omega+\ddc\varphi_0=\sum_{i=1}^rm_i[D_i]+R_{\rm init},
 \qquad m_i>0,
\end{equation}
where $R_{\rm init}$ is a closed positive $(1,1)$-current on $X$ with zero
generic Lelong number along each $D_i$.  Denote by $(\varphi_t)$ the
maximal weak solution from $\varphi_0$.  If $r\ge1$, then
\begin{equation}\label{eq:real-persistence-bound}
 \omega+\ddc\varphi_t\ge\sum_{i=1}^r(m_i-2t)_+[D_i]
 \qquad(t>0).
\end{equation}
Here $r_+:=\max\{r,0\}$, and $t$ lies in the interval of existence.
For $r=0$, the conclusion is simply $\omega+\ddc\varphi_t\ge0$.
\end{lemma}

\begin{proof}
The empty-family case is the positivity of the maximal weak flow.
Assume $r\ge1$.  Fix $i$ and a positive time $t$, and let
$\mu_i(t)$ be the generic Siu
coefficient of $T_t:=\omega+\ddc\varphi_t$ along $D_i$.  In the Siu
decomposition
\[
 T_t=\mu_i(t)[D_i]+R_{t,i},
\]
the residual current $R_{t,i}$ has generic Lelong number zero along
$D_i$ by the definition of $\mu_i(t)$; the same is assumed for
$R_{\rm init}$.  For each integer $p\ge1$, Siu's analyticity
theorem \cite{Siu74} (see also \cite{Dem-book}*{Chapter~III,
Corollary~8.5}) shows that
\[
 \{y\in D_i:\nu_y(R_{\rm init})\ge p^{-1}\},\qquad
 \{y\in D_i:\nu_y(R_{t,i})\ge p^{-1}\}
\]
are proper analytic subsets of $D_i$.  Their union over $p$, together with
the singular locus of $D_i$ and the intersections with the other
divisors, is a countable union of closed nowhere-dense subsets.  The Baire
category theorem therefore provides
$x_{i,t}\in D_i\setminus\bigcup_{\ell\ne i}D_\ell$ at which $D_i$ is
smooth and both residual point Lelong numbers vanish.  The point is
chosen separately for each fixed pair $(i,t)$; no common point for all
times is required.  When $\dim X=1$, $D_i$ is a point and the proper
analytic exceptional subsets are empty, so the same argument applies.
By \cref{lem:SNC},
\[
 \arn(\varphi_0,x_{i,t})=m_i/2,
 \qquad \arn(\varphi_t,x_{i,t})=\mu_i(t)/2.
\]
The pointwise lower estimate in
\cite{DGL26geometry}*{Theorem~1.9}, transported by the positive-scaling
identity in \cref{sec:conventions} if necessary, gives
$\mu_i(t)/2\ge\max\{m_i/2-t,0\}$, or equivalently
$\mu_i(t)\ge(m_i-2t)_+$.  Since this holds for every $i$ and every
fixed positive $t$, the Siu decomposition of $T_t$ proves
\eqref{eq:real-persistence-bound}.
\end{proof}

For the Hirzebruch surface in \cref{thm:robust-counterexample}, the class
$\{S\}+e\mathfrak f$ is nef: its line bundle is globally generated, as
recalled in \cref{sec:Hirzebruch}.  The decomposition
\[
 \{S\}+\mathfrak f
 =\frac1e(\{S\}+e\mathfrak f)+\left(1-\frac1e\right)\{S\}
\]
is its Zariski decomposition.  Indeed, the first summand is nef,
its intersection with $S$ is zero, and the second is effective with
negative definite intersection matrix $(-e)$.  These properties
characterize the surface Zariski decomposition.

\begin{proof}[Proof of Theorem~\ref{thm:robust-counterexample}]
For $0<t<ea/2$, \cref{lem:real-persistence} gives
$T_t:=\omega+\ddc\varphi_t\ge(a-2t)_+[S]+(b-2t)_+[F_0]$.
Hence the signed extraction
\[
 \begin{aligned}
 Q_t&:=T_t-(a-2t)[S]-(b-2t)[F_0]\\
 &=\bigl(T_t-(a-2t)_+[S]-(b-2t)_+[F_0]\bigr)
   +(2t-a)_+[S]+(2t-b)_+[F_0]\ge0
 \end{aligned}
\]
is a closed positive current and, by \eqref{eq:robust-initial}, has
class $2t(\{S\}+\mathfrak f)$.  The displayed Zariski decomposition and homogeneity of minimal multiplicities give
\[
 N\bigl(2t(\{S\}+\mathfrak f)\bigr)
 =2t\left(1-\frac1e\right)S.
\]
Boucksom's domination \eqref{eq:Bou-domination} therefore yields
\[
 \omega+\ddc\varphi_t
 \ge\left(a-\frac{2t}{e}\right)[S]+(b-2t)[F_0].
\]
At a point of $S\setminus F_0$, the current-to-Arnold implication in
\cref{lem:SNC} proves the first inequality in
\eqref{eq:robust-lower-bound}.  The same lemma gives
$\arn(\varphi_0,x)=a/2$.  Since $e>1$ and $t<ea/2$,
$a/2-t/e$ is positive and strictly greater than $a/2-t$; hence
$a/2-t/e>\max\{a/2-t,0\}$.
\end{proof}

\section{Exact local invariants for divisorial data}
\label{sec:surface-exactness}

We prove \cref{thm:surface-exactness} together with the following
description of the actual residual current.

\begin{proposition}[The actual residual current]
\label[proposition]{prop:actual-residual}
Under the hypotheses of \cref{thm:surface-exactness}, set
\begin{equation}\label{eq:surface-coefficients}
 a_i(t):=m_i-2t(1-\eta_i),\qquad
 T_t:=\omega_t=\omega+\ddc\varphi_t.
\end{equation}
For every $0<t<\tau_Z$, the actual residual current satisfies
\begin{equation}\label{eq:actual-residual}
 \mathcal R_t:=T_t-\sum_i a_i(t)[D_i]\ge0,\qquad
 \{\mathcal R_t\}=2tP,\qquad
 \nu_x(\mathcal R_t)=0\quad(x\in X).
\end{equation}
This current has full Monge--Amp\`ere mass in $2tP$, and
\begin{equation}\label{eq:surface-mass}
 \int_X\langle T_t^n\rangle
 =\int_X\langle\mathcal R_t^n\rangle=(2t)^nP^n,
\end{equation}
where $\langle\cdot\rangle$ denotes the non-pluripolar product.
The point Lelong numbers of the evolving current are therefore
\begin{equation}\label{eq:surface-point-lelong}
 \nu_x(T_t)=\sum_{i:\,x\in D_i}a_i(t).
\end{equation}
\end{proposition}

At an SNC crossing, these numbers add, whereas the Arnold multiplicity
in \eqref{eq:surface-exact-profile} is a maximum.

Extracting the uncorrected linear divisor terms gives
$Q_t:=T_t-\sum_i(m_i-2t)[D_i]$ in $2t\beta$;
extracting the corrected terms gives $\mathcal R_t$
in $2tP$.  Both class families consist of positive multiples of a
fixed class.  This homogeneity allows one elliptic solution, whose
fixed-part remainder $R_P$ has class $P$, to generate the comparison family.

\subsection{Automatic positivity and the auxiliary divisor}
Taking cohomology in \eqref{eq:pure-initial} gives
$\{\omega\}=\sum_i m_i\{D_i\}$.  This also implies that $X$ is
projective: approximate the positive weights by positive rational
numbers within the open K\"ahler cone, clear denominators, and apply
Kodaira's embedding theorem to the resulting integral K\"ahler class.

Put $M=\max_i m_i$.  The current
\[
 T=\frac1M\omega+\sum_i\left(1-\frac{m_i}{M}\right)[D_i]
\]
represents $\beta$ and dominates a K\"ahler form.  Thus $\beta$ is big.
Divisorial domination $T\ge[N(\beta)]$ shows that $N$ is supported on
$D$ and gives
\begin{equation}\label{eq:eta-bound}
 0\le\eta_i\le1-\frac{m_i}{M}<1.
\end{equation}
Thus the decomposition has the form
\begin{equation}\label{eq:intro-good-ZD}
 \beta=P+\{N\},\qquad N=\sum_i\eta_iD_i,\qquad 0\le\eta_i<1.
\end{equation}
In particular, $m_i=M$ implies $\eta_i=0$, so the corresponding
term in \eqref{eq:surface-exact-profile} decays at unit speed;
terms with smaller initial coefficients can decay more slowly.
The same bounds give $\tfrac12\min_i m_i\le\tau_Z\le M/2$.

The smooth semipositive representative makes $P$ nef.  Boucksom's
volume identity \cite{Bou04}*{Proposition~3.20} therefore gives
$P^n=\vol(P)=\vol(\beta)>0$ in every dimension.
Define the effective real divisor
\begin{equation}\label{eq:Eaux-support}
 E_{\rm aux}=\sum_i\left(1-\eta_i-\frac{m_i}{2M}\right)D_i.
\end{equation}
Each coefficient is at least $m_i/(2M)>0$, and
\begin{equation}\label{eq:Eaux-class}
 \Supp E_{\rm aux}=\Supp D,\qquad
 P-\{E_{\rm aux}\}=\frac1{2M}\{\omega\}.
\end{equation}
Hence the required auxiliary K\"ahler class and SNC union of supports
are automatic.  Smooth semipositivity of a representative of $P$ remains
the additional hypothesis of the theorem.

For the Arnold upper bound, we construct a candidate $\Psi_t\le\varphi_t$;
monotonicity then gives
$\lambda(\varphi_t,x)\le\lambda(\Psi_t,x)$.

\subsection{The fixed-part and comparison lemmas}

The full-mass framework of \cite{BEGZ10} separates logarithmic poles
forced by a class from a residual potential.  Generic divisorial
coefficients of full-mass currents are identified in
\cite{DFT17}*{Theorem~3.1}; the pointwise statement used below is
\cite{DDL18}*{Theorem~1.1}.  The distinction matters because the exact
profile is asserted at every point, including crossings.

For a big class $\alpha=\{\theta\}$ and
$u\in\PSH(X,\theta)$, the non-pluripolar product
$\langle(\theta+\ddc u)^n\rangle$ is the measure constructed in
\cite{BEGZ10}; it does not charge pluripolar sets, which are locally
contained in the $-\infty$ locus of a plurisubharmonic function not
identically $-\infty$.  Its maximal possible mass is $\vol(\alpha)$.  We write $u\in\E(X,\theta)$ when
$u\in\PSH(X,\theta)$ and
\begin{equation}\label{eq:full-mass}
 \int_X\langle(\theta+\ddc u)^n\rangle=\vol(\alpha).
\end{equation}
We denote by $\E^1(X,\theta)$ the finite-energy subclass, defined using
energy relative to a potential with minimal singularities as in
\cite{BEGZ10}.  Below we only use the inclusion
$\E^1(X,\theta)\subset\E(X,\theta)$.

\begin{lemma}\label[lemma]{lem:fixed-part}
Let $X$ be a compact K\"ahler manifold of dimension $n\ge1$, and let
$\beta=P+\{N\}$ be a big divisorial Zariski decomposition.  Suppose $P$
has a smooth semipositive representative $\theta_P$.  Let $T$ be a closed
positive current with $\{T\}=\beta$ and full Monge--Amp\`ere mass, and
write $T=[N]+R$, where $R$ is a closed positive current with $\{R\}=P$.
Then
\begin{equation}\label{eq:fixed-product}
 \langle T^n\rangle=\langle R^n\rangle
 \quad\text{and}\quad R\text{ has full mass in }P.
\end{equation}
Writing $R=\theta_P+\ddc\widehat u$, one has
$\nu_x(\widehat u)=0$ for every $x\in X$.
\end{lemma}

\begin{proof}
By locality of the non-pluripolar product, the two measures agree off
$\Supp N$, and neither measure charges that pluripolar set
\cite{BEGZ10}*{Proposition~1.4}.  This proves the measure identity.
Boucksom's volume formula gives $\vol(\beta)=\vol(P)$
\cite{Bou04}*{Proposition~3.20}; the smooth semipositive representative
makes $P$ nef, so $\vol(P)=P^n$.  Integrating the measure identity and using the full mass of
$T$ proves that $R$ has full mass.

The class $P$ is big, since $\vol(P)=\vol(\beta)>0$, and $R$ has
full mass in $P$.  The smooth semipositive representative $\theta_P$
has minimal singularities in this class.  Applying
\cite{DDL18}*{Theorem~1.1} directly in $P$ therefore gives
$\nu_x(R)=\nu_x(\theta_P)=0$ for every $x\in X$, including points at
which components of $N$ meet.  Since $\theta_P$ is smooth,
$\nu_x(\widehat u)=0$ at every point as well.
\end{proof}

\begin{lemma}
\label[lemma]{lem:finite-comparison}
Let $(X,\omega)$ be a compact K\"ahler manifold of dimension $n\ge1$,
let $dV_X$ be a smooth positive volume form, let
$\varphi_0\in\PSH(X,\omega)$, and denote by
$(\varphi_t)$ its maximal weak solution.  Let $\mathcal G$ be a nonempty finite set of
prime divisors such that $Z:=\sum_{G\in\mathcal G}G$ is SNC, and set
$\Omega=X\setminus\Supp Z$.  Fix a comparison cutoff $\tau_*>0$.  Suppose
$\Psi_t\in\PSH(X,\omega)$ for $0\le t\le\tau_*$,
$\Psi_0=\varphi_0$, and:
\begin{enumerate}[label=\textup{(C\arabic*)}]
\item $\Psi$ is smooth on $(0,\tau_*]\times\Omega$ and
 \begin{equation}\label{eq:candidate-flow}
 (\omega+\ddc\Psi_t)^n=\exp(\dot\Psi_t)dV_X
 \quad\text{there};
 \end{equation}
\item for each $G\in\mathcal G$, choose a defining section $s_G$ and a
      smooth Hermitian metric $h_G$, and set
      $\ell_G:=\log|s_G|_{h_G}$.  There are continuous functions
      $a_G:(0,\tau_*]\to(0,+\infty)$ and functions $u_t$ on $X$ such that
      $\ell_G\le0$ and
      \[
       \Psi_t=\sum_{G\in\mathcal G}a_G(t)\ell_G+u_t .
      \]
      For every $0<\delta<\tau_*$,
      \[
       \min_{\substack{G\in\mathcal G\\t\in[\delta,\tau_*]}}a_G(t)>0,
       \qquad
       \sup_{[\delta,\tau_*]\times X}u_t<+\infty;
      \]
\item $\Psi_t\to\varphi_0$ in $L^1(X)$ as $t\downarrow0$.
\end{enumerate}
Then, for every smooth strictly $\omega$-psh decreasing approximation
$\varphi_{0,j}\downarrow\varphi_0$ and its smooth flow
$\varphi_{t,j}$,
\begin{equation}\label{eq:comparison-conclusion}
 \Psi_t\le\varphi_{t,j}\quad(0<t\le\tau_*),
 \qquad\Psi_t\le\varphi_t.
\end{equation}
\end{lemma}

\begin{proof}
Use the logarithmic weights from \textup{(C2)} and put
$H_Z=\sum_{G\in\mathcal G}\ell_G$.  For each sufficiently large $L>0$ such that $-L$ is a regular
value of $H_Z$, set $\Omega_L=\{H_Z>-L\}\Subset\Omega$; this domain has smooth
boundary.  Fix $j\in\mathbb N$, $\varepsilon>0$, and a target
$(t_*,x_*)\in(0,\tau_*]\times\Omega$.

Condition \textup{(C3)} and the Hartogs lemma for $\omega$-psh functions
give
\begin{equation}\label{eq:comparison-initial-upper}
 \limsup_{t\downarrow0}\sup_X(\Psi_t-\varphi_{0,j})
 \le \sup_X(\varphi_0-\varphi_{0,j})\le0.
\end{equation}
Indeed, failure of this inequality would give a sequence $t_\nu\downarrow0$
contradicting the $L^1$ Hartogs lemma.  The smooth flow from
$\varphi_{0,j}$ satisfies
$\varphi_{t,j}\to\varphi_{0,j}$ in $C^0(X)$.  Choose
$0<\delta<t_*$ so small that
\[
 \sup_X(\Psi_\delta-\varphi_{0,j})<\frac{\varepsilon}{4},
 \qquad
 \|\varphi_{\delta,j}-\varphi_{0,j}\|_{C^0(X)}
 <\frac{\varepsilon}{4}.
\]
\begin{equation}\label{eq:initial-boundary}
 \Psi_\delta-\varphi_{\delta,j}-\varepsilon(\delta+1)<0
 \quad\text{on }X.
\end{equation}

After fixing $\delta$, define
$c_\delta:=\min_{G\in\mathcal G,\,t\in[\delta,t_*]}a_G(t)>0$.
Condition \textup{(C2)}, together with $\ell_G\le0$, gives a constant
$C_{\delta,t_*}$ such that
\begin{equation}\label{eq:lateral-decay}
 \Psi_t\le c_\delta H_Z+C_{\delta,t_*}
 \qquad(\delta\le t\le t_*).
\end{equation}
The smooth function $\varphi_{t,j}$ has a uniform lower bound on the compact
cylinder $[\delta,t_*]\times X$.  Choose $L$ sufficiently large with $-L$ a regular value of $H_Z$,
also satisfying $x_*\in\Omega_L$, so that on $H_Z=-L$ the right-hand
side of \eqref{eq:lateral-decay}, minus this lower bound and
$\varepsilon(\delta+1)$, is negative.  Then
\begin{equation}\label{eq:lateral-boundary}
 \Psi_t-\varphi_{t,j}-\varepsilon(t+1)<0
 \quad\text{on }[\delta,t_*]\times\partial\Omega_L.
\end{equation}

Set $w=\Psi_t-\varphi_{t,j}-\varepsilon(t+1)$.  It is negative on the
initial and lateral boundary.  If it had a positive maximum at
$(s,y)\in(\delta,t_*]\times\Omega_L$, then its spatial complex Hessian
would be nonpositive.  If $s<t_*$, its time derivative would vanish; if
$s=t_*$, its left time derivative would be nonnegative.  In both cases,
\begin{equation}\label{eq:max-principle-derivatives}
 \dot\Psi-\dot\varphi_{t,j}\ge\varepsilon,
 \qquad \ddc\Psi\le\ddc\varphi_{t,j}.
\end{equation}
The form $\omega+\ddc\Psi_t$ is semipositive and, by
\eqref{eq:candidate-flow}, has positive determinant, hence is positive
definite on $\Omega$.  Determinant monotonicity and the two flow equations
give $\dot\Psi\le\dot\varphi_{t,j}$ at $(s,y)$, contradicting
\eqref{eq:max-principle-derivatives}.  Thus, at the fixed target,
$\Psi_{t_*}\le\varphi_{t_*,j}+\varepsilon(t_*+1)$.  Let
$\varepsilon\downarrow0$, and then use the pointwise monotone convergence
$\varphi_{t,j}\downarrow\varphi_t$ as $j\to\infty$.  Since the target in
$\Omega$ was arbitrary, both inequalities hold there.  At a point of
$\Supp Z$ and any fixed $t>0$, condition \textup{(C2)} gives a strictly
positive coefficient for at least one logarithmic pole through that point;
hence $\Psi_t=-\infty$ there, so the comparison holds globally.
\end{proof}

\subsection{Proof of the exactness theorem}

\begin{proof}[Proof of Theorem~\ref{thm:surface-exactness}
and Proposition~\ref{prop:actual-residual}]
Fix a smooth semipositive representative $\theta_P$ of $P$.

\textit{Step 1. The lower bound from the class.}
At every positive time, \cref{lem:real-persistence} gives
$T_t\ge\sum_i(m_i-2t)_+[D_i]$.  Thus the signed extraction
\[
 \begin{aligned}
 Q_t&:=T_t-\sum_i(m_i-2t)[D_i]=\left(T_t-\sum_i(m_i-2t)_+[D_i]\right)
   +\sum_i(2t-m_i)_+[D_i]\ge0
 \end{aligned}
\]
is positive, including when some uncorrected coefficients are negative.
Taking cohomology in \eqref{eq:pure-initial}, its class is
$\{Q_t\}=2t\sum_i\{D_i\}=2t\beta$.  Homogeneity of minimal multiplicities
gives $N(2t\beta)=2tN(\beta)=2tN$, so \eqref{eq:Bou-domination} yields
\begin{equation}\label{eq:general-lower-current}
 \omega+\ddc\varphi_t
 \ge\sum_i(m_i-2t)[D_i]+2t[N].
\end{equation}

The coefficients on the right are $a_i(t)$.  Thus the actual remainder
$\mathcal R_t:=T_t-\sum_i a_i(t)[D_i]$ is positive with class $2tP$.
For $0<t<\tau_Z$, every $a_i(t)$ is positive.  We will compare its
potential with an elliptic full-mass potential in the same class.

\smallskip
\noindent\textit{Step 2. Metric normalization and elliptic existence.}
Let $s_i$ be the canonical defining section of $\mathcal O_X(D_i)$.
Choose a smooth Hermitian metric $h_i$ on this line bundle, write
$\ell_i=\log|s_i|_{h_i}$, and let $\theta_i$ be its smooth normalized
curvature form from \eqref{eq:PL}.  The $\ddc$-lemma gives a smooth real function $f$ with
$\omega-\sum_i m_i\theta_i=\ddc f$.  Replacing the squared metric
$h_1$ by $e^{-2f/m_1}h_1$ changes $\ell_1$ to $\ell_1-f/m_1$ and
$\theta_1$ to $\theta_1+m_1^{-1}\ddc f$.
Consequently we may arrange the exact smooth-form identities
\begin{equation}\label{eq:metric-normalization-general}
 \omega=\sum_i m_i\theta_i,
 \qquad \varphi_0=\sum_i m_i\ell_i+C_0,
\end{equation}
for some $C_0\in\mathbb R$.  Indeed, after the first normalization,
the current identity \eqref{eq:pure-initial} gives
$\ddc(\varphi_0-\sum_i m_i\ell_i)=0$; the difference is pluriharmonic on
the compact connected manifold $X$ and hence constant.
By rescaling each $h_i$ by a positive constant and adjusting $C_0$, we
also arrange $\ell_i\le0$ on $X$.  These constant rescalings do not change
the curvature forms.

The class $\beta$ is big by the construction above.  Put
$\theta_D=\sum_i\theta_i$.  The existence and uniqueness assertion of
\cite{DGL26geometry}*{Theorem~3.4} applies to a smooth closed representative
$\theta$ of a big class, an SNC divisor with defining sections $s_i$
and smooth metrics $h_i$, and a smooth positive volume form $dV$.
It gives a unique $v\in\E^1(X,\theta)$ satisfying
$\langle(\theta+\ddc v)^n\rangle=e^v dV/\prod_i|s_i|_{h_i}^2$.
Only existence and uniqueness are used here; regularity is established
separately after removing the fixed part.
Apply it with $\theta=2\theta_D$ and $dV=2^n dV_X$.  It produces a unique
$v\in\E^1(X,2\theta_D)$ satisfying
\[
 \left\langle(2\theta_D+\ddc v)^n\right\rangle
 =\frac{2^n\exp(v)}{\prod_i|s_i|_{h_i}^2}\,dV_X.
\]
Set $\rho=v/2$, equivalently $v=2\rho$.  Since
$2\theta_D+\ddc v=2(\theta_D+\ddc\rho)$, full mass and
$\mathcal E^1$ membership pass through this scaling.  Thus
$\rho\in\E^1(X,\theta_D)$ is the unique solution of
\begin{equation}\label{eq:rho-general}
 \left\langle(\theta_D+\ddc\rho)^n\right\rangle
 =\frac{\exp(2\rho)}{\prod_i|s_i|_{h_i}^2}\,dV_X.
\end{equation}
Since $\E^1(X,\theta_D)\subset\E(X,\theta_D)$, the current
$T_\rho:=\theta_D+\ddc\rho$ is positive and has full mass in $\beta$.

\smallskip
\noindent\textit{Step 3. Removal of the fixed part and regularity.}
Set $\ell_N=\sum_i\eta_i\ell_i$ and
$\theta_N=\sum_i\eta_i\theta_i$.  Boucksom's domination gives
$T_\rho\ge[N]$.  Define $\widehat\rho:=\rho-\ell_N$ off $\Supp N$ and use its
quasi-psh extension across $N$, supplied by the positive residual current.
Then
\begin{equation}\label{eq:rho-general-split}
 \begin{gathered}
  \rho=\ell_N+\widehat\rho,\qquad
  R_P:=T_\rho-[N]=\theta_D-\theta_N+\ddc\widehat\rho\ge0,
  \qquad \{R_P\}=P,\\
  \nu_x(\widehat\rho)=0\qquad\text{for every }x\in X.
 \end{gathered}
\end{equation}
Here the last assertion and the fact that $R_P$ has full mass in $P$ follow
from \cref{lem:fixed-part}.

Choose a smooth function $f_P$ such that
$\theta_P=\theta_D-\theta_N+\ddc f_P$ and put
$u:=\widehat\rho-f_P$.  Then $\theta_P+\ddc u=R_P$.  Define
\[
 \Delta_{\rm BG}:=\sum_i(1-\eta_i)D_i.
\]
Its reduced support is $D$, since $1-\eta_i>0$.
Removing the fixed divisor from
\eqref{eq:rho-general} gives
\begin{equation}\label{eq:BG-general-equation}
 \left\langle(\theta_P+\ddc u)^n\right\rangle
 =\frac{\exp(2u)}
 {\displaystyle\prod_i
 |s_i|_{h_i}^{2(1-\eta_i)}}\,dV_P,
 \qquad dV_P=\exp(2f_P)dV_X.
\end{equation}
The form $dV_P$ is a smooth positive volume form.  To match
\cite{BG14}*{Theorem~4.6}, note that its squared-norm metric convention
uses $\ddc_{\rm BG}=\tfrac12\ddc$: in that convention
$\ddc_{\rm BG}\log|z_1|^2=[z_1=0]$.
The weight $v_{\rm BG}:=2u$ therefore satisfies
\[
 \theta_P+\ddc_{\rm BG}v_{\rm BG}=R_P,
 \qquad
 \left\langle(\theta_P+\ddc_{\rm BG}v_{\rm BG})^n\right\rangle
 =\frac{\exp(v_{\rm BG})}
 {\prod_i|s_i|_{h_i}^{2(1-\eta_i)}}\,dV_P.
\]
All hypotheses of \cite{BG14}*{Theorem~4.6} now hold: $\theta_P$ is
smooth semipositive with $\int_X\theta_P^n=P^n>0$, $R_P$ has full mass in $P$,
$dV_P$ is smooth and positive, and the coefficients $1-\eta_i$ lie in
$(0,1]$.  The auxiliary divisor \eqref{eq:Eaux-support} satisfies
\eqref{eq:Eaux-class}; the union of its support with that of
$\Delta_{\rm BG}$ is the SNC support of $D$.  The cited theorem permits
common components in these divisors and requires only this union of
supports to be SNC.  It follows that $u$, and hence $\widehat\rho$, is
smooth off $D$:
\begin{equation}\label{eq:rho-general-smooth}
 \widehat\rho\in C^\infty\bigl(X\setminus(\Supp D)\bigr).
\end{equation}

\smallskip
\noindent\textit{Step 4. Construction and comparison of the parabolic candidate.}
Fix a comparison cutoff $0<\tau_*<\tau_Z$.  For $0<t\le\tau_*$, define
\begin{equation}\label{eq:general-candidate}
 \Psi_t=\sum_i a_i(t)\ell_i+2t\widehat\rho+n(t\log(2t)-t)+C_0,
\end{equation}
and set $\Psi_0:=\varphi_0$.  We use the continuous extension
$t\log(2t)|_{t=0}=0$.
Equations \eqref{eq:metric-normalization-general} and
\eqref{eq:rho-general-split} give
\begin{equation}\label{eq:general-candidate-current}
 \omega+\ddc\Psi_t
 =\sum_i a_i(t)[D_i]+2tR_P\ge0.
\end{equation}
On $\Omega=X\setminus(\Supp D)$,
the current $R_P$ is smooth and semipositive by
\eqref{eq:rho-general-smooth}.  Equation \eqref{eq:rho-general} gives a
strictly positive smooth density for $R_P^n$ there, so $R_P$ is positive
definite on $\Omega$.  Moreover,
$\omega+\ddc\Psi_t=2tR_P$ on $\Omega$ and
\[
 \dot\Psi_t=-2\sum_i\ell_i+2\rho+n\log(2t).
\]
Consequently the following is an equality of smooth positive volume forms:
\begin{equation}\label{eq:general-candidate-flow}
 (\omega+\ddc\Psi_t)^n
 =(2t)^n\frac{\exp(2\rho)}{\prod_i|s_i|_{h_i}^2}dV_X
 =\exp(\dot\Psi_t)dV_X.
\end{equation}
The coefficients in \eqref{eq:surface-coefficients} satisfy
\begin{equation}\label{eq:candidate-divisor-coefficient}
 a_i(t)=m_i-2t(1-\eta_i)\ge m_i-2\tau_*(1-\eta_i)>0.
\end{equation}
Using the previously chosen metrics, we can write
\[
 \Psi_t=\sum_i a_i(t)\ell_i+u_t,
 \qquad
 u_t=2t\widehat\rho+n(t\log(2t)-t)+C_0.
\]
For every $0<\delta<\tau_*$, the preceding estimates give
$\min_{1\le i\le r,\,t\in[\delta,\tau_*]}a_i(t)>0$.  Since a quasi-psh function on
the compact manifold $X$ is bounded above, $\widehat\rho$ is bounded
above; hence
$\sup_{[\delta,\tau_*]\times X}u_t<+\infty$.  Finally,
\eqref{eq:rho-general-split} and local integrability of quasi-psh
functions give the explicit estimate
\[
 \|\Psi_t-\varphi_0\|_{L^1(X)}
 \le2t\!\left(\|\widehat\rho\|_{L^1(X)}+
       \sum_i(1-\eta_i)\|\ell_i\|_{L^1(X)}\right)+O(|t\log t|),
\]
after fixing any smooth volume form for the $L^1$ norm.  Hence
$\Psi_t\to\varphi_0$ in $L^1(X)$ as $t\downarrow0$.
Thus \cref{lem:finite-comparison} applies and yields
\begin{equation}\label{eq:general-candidate-order}
 \Psi_t\le\varphi_t\qquad(0<t\le\tau_*).
\end{equation}

\smallskip
\noindent\textit{Step 5. Full mass of the actual remainder and pointwise invariants.}
By Step~1, the actual remainder $\mathcal R_t$ is positive in $2tP$.
Put $\theta'_P:=\theta_D-\theta_N$ and set
$v_t:=\varphi_t-\sum_i a_i(t)\ell_i$ on $X\setminus D$.
Positivity supplies its global quasi-psh extension, with
\[
 \mathcal R_t=2t\theta'_P+\ddc v_t,\qquad
 2tR_P=2t\theta'_P+\ddc(2t\widehat\rho).
\]
Writing $C_t:=n(t\log(2t)-t)+C_0$, the comparison
\eqref{eq:general-candidate-order} gives
\[
 v_t\ge2t\widehat\rho+C_t.
\]
This inequality extends across $D$ with the quasi-psh extensions.
Thus $\mathcal R_t$ is less singular than the full-mass current $2tR_P$
in the same big class $2tP$.  By \cite{BEGZ10}*{Proposition~2.11(i) and
Proposition~2.14, first assertion}, full-mass currents belong to some
weighted energy class, which contains every less singular positive
current in the same class.  Hence $\mathcal R_t$ has full mass and
\[
 \int_X\langle\mathcal R_t^n\rangle=\vol(2tP)=(2t)^nP^n.
\]
Since $T_t$ and $\mathcal R_t$ agree off $D$ and neither non-pluripolar
product charges $D$, locality gives
$\langle T_t^n\rangle=\langle\mathcal R_t^n\rangle$, proving
\eqref{eq:surface-mass}.

The smooth semipositive representative $2t\theta_P$ has minimal
singularities in $2tP$.  Applying \cite{DDL18}*{Theorem~1.1(i)} to
the actual full-mass current $\mathcal R_t$ now gives
$\nu_x(\mathcal R_t)=\nu_x(2t\theta_P)=0$ at every $x\in X$.
Together with Step~1, this proves \eqref{eq:actual-residual}.
Additivity of Lelong numbers in
$T_t=\sum_i a_i(t)[D_i]+\mathcal R_t$ proves
\eqref{eq:surface-point-lelong}.

Near a fixed $x$, relabel the components through $x$ as
$D_1,\ldots,D_s$ and choose adapted coordinates $D_j=\{z_j=0\}$.
Choose a psh potential $r_t$ of $\mathcal R_t$ on this chart.  Then
\begin{equation}\label{eq:actual-residual-potential}
 \varphi_t=\sum_{j=1}^s a_j(t)\log|z_j|+r_t+h_t,
 \qquad h_t\text{ smooth},\qquad \nu_x(r_t)=0.
\end{equation}
Applying \cref{lem:SNC} proves \eqref{eq:surface-exact-profile}.
Since $\tau_*<\tau_Z$ was arbitrary, this proves
\cref{thm:surface-exactness,prop:actual-residual} on $(0,\tau_Z)$.
\end{proof}

\begin{remark}[Scope of the conclusion]
The full-mass conclusion concerns $\mathcal R_t$ in the class $2tP$;
it follows from the comparison of singularities, rather than from zero
Lelong numbers alone.  It is not a full-mass assertion for $T_t$ in
$\{\omega\}$.  We do not assert that the actual residual potentials
are locally bounded, or that $\Psi_t=\varphi_t$; full-mass potentials
can be unbounded.  All formulas are restricted to
$0<t<\tau_Z$, while every corrected coefficient $a_i(t)$ is positive.
The endpoint and later times are not covered, and $\tau_Z$ is not
asserted to be an optimal smoothing time.
\end{remark}

\begin{remark}[Why the present argument does not cover nonhomothetic families]
\label[remark]{rem:nonhomothetic-obstruction}
For an affine family of residual classes
$\alpha_t=\alpha_0+t\gamma$ with
$\alpha_0,\gamma\in H^{1,1}(X,\mathbb R)$ and $\alpha_0\ne0$, the classes
need not be homothetic, that is, positive multiples of a single fixed
class.  Even if each has a divisorial Zariski decomposition
$P_t+\{N_t\}$ with $P_t$ admitting a smooth semipositive representative
and an elliptic full-mass potential $u_t$ smooth off the relevant divisor,
this only gives spatial regularity at each time.  Even if sufficient
time differentiability is further assumed, substitution into a parabolic
candidate creates $\partial_tu_t$, which the fixed-time elliptic equation
does not control.  Additional parabolic input is therefore needed.
\end{remark}

\section{Hirzebruch surfaces and exact local multiplicities}
\label{sec:Hirzebruch}

Fix an integer $e\ge2$.  Using Grothendieck's quotient convention for
projective bundles, let
\[
 \pi_e:X=\mathbb F_e=\PP_{\PP^1}
 (\mathcal O_{\PP^1}\oplus\mathcal O_{\PP^1}(-e))\longrightarrow\PP^1.
\]
Let $S\subset X$ be the section induced by the quotient onto
$\mathcal O_{\PP^1}(-e)$.  For $p\in\PP^1$, let $F_p=\pi_e^{-1}(p)$ be
the fibre divisor and put
$\mathfrak f:=\{F_p\}\in H^{1,1}(X,\mathbb R)$.  Then
\begin{equation}\label{eq:Hirz-intersections}
 \{S\}^2=-e,\qquad \{S\}\cdot\mathfrak f=1,
 \qquad\mathfrak f^2=0.
\end{equation}
To verify the cones, write the class of an irreducible curve $C$ as
$x\{S\}+y\mathfrak f$.  If $C\ne S$, then
$x=\{C\}\cdot\mathfrak f\ge0$ and
$y-ex=\{C\}\cdot\{S\}\ge0$; the curve $S$ itself supplies the remaining extremal
ray.  Hence the closed effective cone is generated by $\{S\}$ and
$\mathfrak f$.  Duality on a surface shows that the nef cone is generated
by $\mathfrak f$ and $\{S\}+e\mathfrak f$, and the Nakai--Moishezon
criterion gives
\begin{equation}\label{eq:Hirz-Kahler-cone}
 x\{S\}+y\mathfrak f\text{ is K\"ahler}
 \quad\Longleftrightarrow\quad x>0\text{ and }y>ex.
\end{equation}
See also \cite{Hartshorne}*{Chapter~V, Propositions~2.3 and~2.20}.

Fix one fibre $F$ and set
\[
 L:=\mathcal O_X(S+eF)
 =\mathcal O_X(1)\otimes\pi_e^*\mathcal O_{\PP^1}(e).
\]
Here $\mathcal O_X(1)$ is the tautological quotient line bundle.
Then $\pi_{e*}L\simeq
\mathcal O_{\PP^1}(e)\oplus\mathcal O_{\PP^1}$, and the tautological
evaluation map $\pi_e^*\pi_{e*}L\twoheadrightarrow L$ is surjective.
Thus $L$ is globally generated
\cite{Hartshorne}*{Chapter~V, Exercise~2.11}.
The complete linear system $|S+eF|$ is base-point-free, and its associated
morphism contracts $S$.  Pulling back a suitably normalized
Fubini--Study form gives a smooth semipositive representative of the class
$\{S\}+e\mathfrak f$, whose square is $e>0$.

For an integer $k\ge1$,
define the class
$\beta_k:=\{S\}+k\mathfrak f\in H^{1,1}(X,\mathbb R)$.  Its divisorial
Zariski decomposition is $\beta_k=P_k+\{N_k\}$, where $P_k$ is the nef
positive-part class and $N_k$ is the effective real negative divisor given
by
\begin{equation}\label{eq:Hirz-ZD-all-k}
 P_k=\begin{cases}
 \dfrac{k}{e}(\{S\}+e\mathfrak f),&1\le k\le e,\\
 \{S\}+k\mathfrak f,&k>e,
 \end{cases}
 \qquad N_k=\left(1-\frac{k}{e}\right)_+S.
\end{equation}
For $k<e$, the displayed $P_k$ is nef, $N_k\ge0$, $P_k\cdot\{S\}=0$,
$P_k^2=k^2/e>0$, and the
intersection matrix of $N_k$ is $(-e)$; these facts characterize the
Zariski decomposition.  For $k=e$ the negative part is zero and $P_k$ has
the smooth semipositive big representative just constructed.  For $k>e$,
\eqref{eq:Hirz-Kahler-cone} shows that $P_k=\beta_k$ is K\"ahler. 
Distinct fibres are disjoint and meet $S$ transversely, so
$S+\sum_iF_i$ has SNC support.

\begin{theorem}\label[theorem]{thm:Hirz-exact}
Fix integers $e\ge2$ and $k\ge1$, and let $F_1,\ldots,F_k$ be distinct
fibres of $\mathbb F_e$.
Let $a,b_1,\ldots,b_k>0$, put $B=\sum_{i=1}^k b_i$, and assume $B>ea$.  Fix a
K\"ahler form $\omega$ and a smooth positive volume form $dV_X$, with
\begin{equation}\label{eq:Hirz-class}
 \{\omega\}=a\{S\}+B\mathfrak f,
\end{equation}
and choose $\varphi_0\in\PSH(\mathbb F_e,\omega)$ such that
\begin{equation}\label{eq:Hirz-initial}
 \omega+\ddc\varphi_0=a[S]+\sum_{i=1}^kb_i[F_i].
\end{equation}
Set
\begin{equation}\label{eq:Hirz-time-slope}
 \kappa_{e,k}=\min\left\{\frac{k}{e},1\right\},
 \qquad \tau_Z=\frac12\min\{a/\kappa_{e,k},b_1,\ldots,b_k\}.
\end{equation}
Then for $0<t<\tau_Z$, the maximal weak solution $(\varphi_t)$ satisfies
\begin{equation}\label{eq:Hirz-exact-formula}
 \arn(\varphi_t,x)=
 \begin{cases}
 a/2-\kappa_{e,k}t,&x\in S\setminus\bigcup_iF_i,\\[1mm]
 b_i/2-t,&x\in F_i\setminus S,\\[1mm]
 \max\{a/2-\kappa_{e,k}t,b_i/2-t\},&x=S\cap F_i,\\[1mm]
 0,&x\notin S\cup\bigcup_iF_i.
 \end{cases}
\end{equation}
\end{theorem}

For $0<t<\tau_Z$, the Arnold multiplicity on
$S\setminus\bigcup_iF_i$ decays at rate $\kappa_{e,k}$, while the
coefficient of $[S]$ in the evolving current decays at rate
$2\kappa_{e,k}$.
The reduced support class $\beta_k=\{S\}+k\mathfrak f$ is non-nef
for $k<e$ and nef for $k\ge e$; the formula allows arbitrary positive
real weights in both cases.  For integral weights, the initial potential is locally
$\log|f|$ plus a smooth function, with $f$ holomorphic, so it has
analytic singularities.  The one-fibre choice in \cref{cor:exact-counterexample} gives a
counterexample in this class.

\begin{proof}[Proof of Theorem~\ref{thm:Hirz-exact}]
The condition $B>ea$ is exactly the K\"ahler-cone condition
\eqref{eq:Hirz-Kahler-cone} for the class in \eqref{eq:Hirz-class}.  The
reduced initial divisor is
$D=S+\sum_{i=1}^kF_i$, its components have respective weights
$m_S=a$ and $m_{F_i}=b_i$, and
$\beta=\{S\}+\sum_i\{F_i\}=\beta_k$.  For $k\le e$,
\eqref{eq:Hirz-ZD-all-k} gives a positive multiple of the smooth
semipositive class constructed above; for $k>e$, the positive part is
K\"ahler.  Together with the SNC support, this verifies the hypotheses
of \cref{thm:surface-exactness}.

Applying \cref{thm:surface-exactness}, the term associated with $S$
in the Arnold-multiplicity formula is
\[
 \frac{a-2t}{2}+t\left(1-\frac{k}{e}\right)_+
 =\frac a2-\kappa_{e,k}t,
\]
because the Zariski coefficient of $S$ is
$\eta_S=(1-k/e)_+$; every fibre has Zariski coefficient zero and hence
term $b_i/2-t$.  The SNC maximum formula in \cref{lem:SNC} now gives
all four cases in \eqref{eq:Hirz-exact-formula}, including the maximum at
$S\cap F_i$ and zero off the divisor.  The comparison interval is precisely
$0<t<\tau_Z=\frac12\min\{a/\kappa_{e,k},b_1,\ldots,b_k\}$;
neither $b_i>a$ nor integrality is required.
\end{proof}

\begin{proof}[Proof of Theorem~\ref{cor:exact-counterexample}]
For $k=1$, $a=1$, and $b_1=e+1$, one has
\[
 B=e+1>e=ea,\qquad
 \tau_Z=\tfrac12\min\{e,e+1\}=\tfrac e2,\qquad
 \kappa_{e,1}=\frac{1}{e}.
\]
On a sufficiently small coordinate neighbourhood $U$, let $f_S$ and
$f_1$ define $S$ and $F_1$.  The local Poincar\'e--Lelong formula and
the initial current identity show that, for some $h\in C^\infty(U)$,
\[
 \varphi_0
 =\log|f_S|+(e+1)\log|f_1|+h
 =\log\bigl|f_Sf_1^{\,e+1}\bigr|+h.
\]
Hence $\varphi_0$ has analytic singularities in the sense of
\eqref{eq:analytic-singularities}.

If $x\in S\setminus F_1$, this local expression and
\cref{lem:SNC} give $\arn(\varphi_0,x)=1/2$.  Theorem~\ref{thm:Hirz-exact}
gives
\[
 \arn(\varphi_t,x)=\frac12-\frac{t}{e}\qquad(0<t<e/2).
\]
For $e\ge2$ and $0<t<e/2$, this is positive and greater than
$1/2-t$, hence greater than
$\max\{\arn(\varphi_0,x)-t,0\}$, proving
\eqref{eq:Q110-counterexample}.
\end{proof}

\section{Nonlocality, multiplier ideals, and quantitative consequences}
\label{sec:applications}

The exact profiles first let us compare flows whose initial positive
currents agree near a point.  We then extract finer integrability information through
multiplier ideals, and finish with persistence loci and a quantitative
comparison with the known upper estimate.

Changing the number of fibre components while preserving their total
weight changes the reduced support class $\beta=\{S\}+k\mathfrak f$.
When all selected fibres avoid a neighbourhood of the chosen point on
$S$, the initial current germ there remains fixed.
Theorem~\ref{thm:parametric-nonlocality} produces arbitrarily many such
flows with the same surface, background form, reference volume form,
and global initial Arnold multiplicity.

\begin{proof}[Proof of Theorem~\ref{thm:parametric-nonlocality}]
Fix an integer $M\ge2$ and a real number $a>0$.  Choose an integer
$e\ge M+2$ and a positive integer $L$ such that
\begin{equation}\label{eq:nonlocal-choices}
 \operatorname{lcm}(1,\ldots,M)\mid L,\qquad L>M,\qquad 2L>e.
\end{equation}
Put $C:=aL$, so that $C>Ma$ and $2C>ea$.
Choose distinct fibres $F_0,F_1,\ldots,F_M$ and a point
$x\in S\setminus\bigcup_{\ell=0}^MF_\ell$.  For $1\le j\le M$, define
\begin{equation}\label{eq:nonlocal-currents}
 T_0^{(j)}
 :=a[S]+C[F_0]+\frac{C}{j}\sum_{\ell=1}^j[F_\ell].
\end{equation}
Every $T_0^{(j)}$ represents $a\{S\}+2C\mathfrak f$, which is K\"ahler by
\eqref{eq:nonlocal-choices}.  Fix a K\"ahler form $\omega$ in this class
and a single smooth positive volume form $dV_X$.  The
$\partial\bar\partial$-lemma gives
$\varphi_0^{(j)}\in\PSH(X,\omega)$ satisfying
\[
 \omega+\ddc\varphi_0^{(j)}=T_0^{(j)}.
\]
The local Poincar\'e--Lelong formula writes $\varphi_0^{(j)}$ as
\[
 a\log\left|f_Sf_0^L\prod_{\ell=1}^{j}f_\ell^{L/j}\right|+h_j,
\]
where the $f$'s are local defining functions and $h_j$ is smooth.
All exponents inside the product are integers, and the common positive
coefficient $a$ is allowed by \eqref{eq:analytic-singularities}.
Thus each $\varphi_0^{(j)}$ has neat analytic singularities.

On a neighbourhood of $x$ avoiding all selected fibres, every current is
exactly $a[S]$, proving \eqref{eq:intro-common-germ}.  The corresponding
potential germs differ by a smooth pluriharmonic function and therefore
have the same singularity type.  Since the divisors are SNC,
\cref{lem:SNC} gives
\[
 \arn(\varphi_0^{(j)})=\frac12\max\left\{a,C,\frac Cj\right\}=\frac C2,
\]
which proves \eqref{eq:intro-common-global-arnold}.

For the $j$-th flow, the initial divisor has $j+1$ fibre components and
total fibre coefficient $2C$.  Moreover,
$C/j>a$ and $ae/(j+1)>a$.  Thus the corrected cutoff for every
flow is greater than $a/2$, and \cref{thm:Hirz-exact} gives
\begin{equation}\label{eq:nonlocal-slopes}
 \arn(\varphi_t^{(j)},x)=\frac a2-\frac{j+1}{e}t,
 \qquad 0<t<a/2.
\end{equation}
The rates $(j+1)/e$, $1\le j\le M$, are pairwise distinct.
\end{proof}

\begin{corollary}
\label[corollary]{cor:no-dimension-rate}
There is no constant $c_2>0$ with the following property: for every
maximal weak solution on a compact K\"ahler surface and every point $x$,
there exists $\varepsilon>0$ such that
\begin{equation}\label{eq:forbidden-rate}
 \arn(\varphi_t,x)\le
 \max\{\arn(\varphi_0,x)-c_2t,0\}
\end{equation}
for $0<t<\varepsilon$.
\end{corollary}

\begin{proof}
Choose an integer $e\ge2$ with $e>1/c_2$.  In
\cref{thm:robust-counterexample}, take $a=1$, $b=e+1$, and
$x\in S\setminus F_0$.  Its current lower bound gives
$\arn(\varphi_t,x)\ge1/2-t/e>1/2-c_2t$ for every
$0<t<\min\{1/2,1/(2c_2)\}$, on which the right-hand side of
\eqref{eq:forbidden-rate} equals $1/2-c_2t$.
For any proposed $\varepsilon>0$, choose $t$ also smaller than
$\varepsilon$ to contradict \eqref{eq:forbidden-rate}.
\end{proof}

The same local residual decomposition determines the multiplier ideals
in the normalization \eqref{eq:multiplier-definition}.

\begin{corollary}[Multiplier ideals for the exact formula]
\label[corollary]{cor:surface-multiplier-ideals}
Under the hypotheses of \cref{thm:surface-exactness}, for $q>0$ and
$0<t<\tau_Z$,
\begin{equation}\label{eq:surface-multiplier-formula}
 \J(q\varphi_t)
 =\mathcal O_X\!\left(-\sum_i
 \left\lfloor q\left[\frac{m_i}{2}-t(1-\eta_i)\right]\right\rfloor D_i\right).
\end{equation}
The wall set, namely the times at which this ideal sheaf is not locally
constant as a function of time, is
\begin{equation}\label{eq:surface-jumping-walls}
 \left(\bigcup_i
 \left\{\frac{m_i/2-\ell/q}{1-\eta_i}:\ell\in\mathbb Z\right\}\right)
 \cap(0,\tau_Z).
\end{equation}
At a wall the ideal has a strict right-hand increase at every smooth
point of an attaining component off the other components.
\end{corollary}

\begin{proof}
By \cref{prop:actual-residual}, local potentials of the actual residual
current have zero point Lelong numbers.  Using the local decomposition
\eqref{eq:actual-residual-potential}, the multiplier-ideal part of
\cref{lem:SNC}, including its integer-endpoint argument, proves
\eqref{eq:surface-multiplier-formula} directly for $\varphi_t$.

By \eqref{eq:intro-good-ZD}, all $1-\eta_i>0$.  Thus the floors are
locally constant away from \eqref{eq:surface-jumping-walls}.
At a listed time $t$, choose a smooth point of an attaining component
$D_i$ off the others and a local defining coordinate $z$.
If $\ell=q a_i(t)/2\in\mathbb Z_{>0}$, then for sufficiently small $h>0$
the local ideals at $t$ and $t+h$ are $(z^\ell)$ and $(z^{\ell-1})$.
The section $z^{\ell-1}$ witnesses their strict inclusion, including
when several components attain a wall simultaneously.  All coefficients
decrease, so the ideal sheaves increase with time; strictness is local
along the attaining components.
\end{proof}

\begin{corollary}[Multiplier-ideal jumping times on Hirzebruch surfaces]
\label[corollary]{thm:multiplier-ideals}
Under the hypotheses of \cref{thm:Hirz-exact}, put
\begin{equation}\label{eq:dynamic-divisor-coefficients}
 c_S(t)=a-2\kappa_{e,k}t,\qquad c_i(t)=b_i-2t.
\end{equation}
For every $q>0$ and $0<t<\tau_Z$,
\begin{equation}\label{eq:multiplier-formula}
 \J(q\varphi_t)=\mathcal O_X\!\left(
 -\lfloor qc_S(t)/2\rfloor S-\sum_i\lfloor qc_i(t)/2\rfloor F_i\right).
\end{equation}
The wall set is
\begin{equation}\label{eq:jumping-walls}
 W_q=\left(
 \left\{\frac{a/2-\ell/q}{\kappa_{e,k}}:\ell\in\mathbb Z\right\}
 \cup\bigcup_i\left\{\frac{b_i}{2}-\frac\ell q:\ell\in\mathbb Z\right\}
 \right)\cap(0,\tau_Z),
\end{equation}
with the right-hand changes described in
\cref{cor:surface-multiplier-ideals}.
\end{corollary}

\begin{proof}
In \eqref{eq:Hirz-ZD-all-k}, $1-\eta_S=\kappa_{e,k}$ and
$\eta_{F_i}=0$.  Substitute in \cref{cor:surface-multiplier-ideals}.
\end{proof}

\begin{remark}[Persistence loci and asymptotic sharpness]
\label[remark]{rem:persistence-sharpness}
For the Hirzebruch flow and $0<t<\tau_Z$, set $c_S(t)=a-2\kappa_{e,k}t$
and $c_{F_i}(t):=c_i(t)=b_i-2t$.  For $\gamma>0$, \eqref{eq:Hirz-exact-formula} gives
\[
 \{x:\arn(\varphi_t,x)\ge\gamma\}
 =\bigcup_{\substack{G\in\{S,F_1,\ldots,F_k\}\\c_G(t)\ge2\gamma}}G.
\]
In the one-fibre family of \cref{cor:exact-counterexample},
denote the negative section and selected fibre on $\mathbb F_e$ by
$S_e$ and $F_{1,e}$, and the corresponding flow by $(\varphi_t^{(e)})$.
For each integer $e\ge2$, choose $x_e\in S_e\setminus F_{1,e}$.
For $0<t<e/2$,
the upper estimate in \eqref{eq:DGL-two-sided} is
$U_e(t)=1/2-t/(e+1)$, and
\[
 U_e(t)-\arn(\varphi_t^{(e)},x_e)=\frac{t}{e(e+1)},\qquad
 \frac{1/2-\arn(\varphi_t^{(e)},x_e)}{1/2-U_e(t)}=\frac{e+1}{e}\longrightarrow1\qquad(e\to\infty).
\]
For the limit, $t>0$ is held fixed and $e>2t$.
This sharpness is asymptotic along a family whose global initial
multiplicity $(e+1)/2$ varies with $e$.
\end{remark}

\section{Product constructions and transfer}
\label{sec:products}

Products transfer the explicit examples and realize maxima of local
profiles.  They also allow smooth stationary factors, whose initial
currents include a smooth positive part and need not satisfy the
pure divisorial hypothesis of \cref{thm:surface-exactness}.  We first
identify the maximal product flow using a separable decreasing smooth
approximation and independence of the maximal weak flow from the chosen
approximation, and then apply this construction.

\begin{theorem}
\label[theorem]{thm:product}
Let $r\ge1$ be an integer.  For $1\le i\le r$, let
$(X_i,\omega_i)$ be a compact K\"ahler manifold
of complex dimension $n_i\ge1$, let $dV_i$ be a smooth positive volume
form, and let $\varphi_{i,t}$ be the maximal weak solution from
$\varphi_{i,0}\in\PSH(X_i,\omega_i)$.  Assume that all factor flows are
defined on a common interval $0\le t<T$.  Put
\begin{equation}\label{eq:product-data}
 X=\prod_{i=1}^rX_i,\qquad
 n=\sum_{i=1}^r n_i.
\end{equation}
For each $i$, let $\pr_i:X\to X_i$ be the $i$-th projection, and define
the K\"ahler form
$\omega:=\sum_{i=1}^r\pr_i^*\omega_i$ on $X$, and normalize the
product volume form by
\begin{equation}\label{eq:product-volume}
 dV=\frac{n!}{\prod_{i=1}^r n_i!}
 \bigwedge_{i=1}^r\pr_i^*dV_i.
\end{equation}
Then, for $0\le t<T$,
\begin{equation}\label{eq:product-potential}
 \Phi_t=\sum_{i=1}^r\pr_i^*\varphi_{i,t}
\end{equation}
is the maximal weak solution from
$\Phi_0=\sum_{i=1}^r\pr_i^*\varphi_{i,0}
\in\PSH(X,\omega)$.  At $x=(x_1,\ldots,x_r)\in X$,
\begin{equation}\label{eq:product-arnold}
 \lct(\Phi_t,x)=\min_{1\le i\le r}\lct(\varphi_{i,t},x_i),
 \qquad
 \arn(\Phi_t,x)=\max_{1\le i\le r}\arn(\varphi_{i,t},x_i).
\end{equation}
\end{theorem}

\begin{proof}
For each $i$, choose smooth strictly $\omega_i$-psh initial data
$\varphi_{i,0,j}\downarrow\varphi_{i,0}$ and denote the corresponding
smooth flows by $\varphi_{i,t,j}$.  Set
\[
 \Phi_{0,j}=\sum_i\pr_i^*\varphi_{i,0,j},
 \qquad
 \Phi_{t,j}=\sum_i\pr_i^*\varphi_{i,t,j}.
\]
The first sequence is a decreasing smooth $\omega$-psh approximation of
$\Phi_0$.  Set
\[
 A_{i,t,j}:=\pr_i^*(\omega_i+\ddc\varphi_{i,t,j}),
\]
a smooth semipositive real $(1,1)$-form on $X$.  Then
$A_{i,t,j}^{\ell}=0$ for $\ell>n_i$.  Hence a multinomial term of total
degree $n=\sum_i n_i$ is nonzero only when the exponent of every
$A_{i,t,j}$ is $n_i$.  Therefore
\begin{equation}\label{eq:product-smooth-flow}
 (\omega+\ddc\Phi_{t,j})^n
 =\frac{n!}{\prod_i n_i!}\bigwedge_i
   \pr_i^*(\omega_i+\ddc\varphi_{i,t,j})^{n_i}
 =\exp\!\left(\sum_i\pr_i^*\dot\varphi_{i,t,j}\right)dV
 =\exp(\dot\Phi_{t,j})dV.
\end{equation}
Smooth uniqueness identifies $\Phi_{t,j}$ with the smooth flow on the
product from $\Phi_{0,j}$.  Passing to the pointwise decreasing limit and
using approximation independence in the maximal-flow construction recalled in
\cref{subsec:maximal-flow} proves
\eqref{eq:product-potential}, because the finite sum satisfies
$\Phi_{t,j}\downarrow\sum_{i=1}^r\pr_i^*\varphi_{i,t}$.

Fix $q>0$.  On a product of small coordinate neighbourhoods $U_i\ni x_i$, Tonelli
gives
\begin{equation}\label{eq:product-tonelli}
 \int_{\prod_iU_i}e^{-q\Phi_t}\bigwedge_i\pr_i^*dV_i
 =\prod_i\int_{U_i}e^{-q\varphi_{i,t}}dV_i.
\end{equation}
Each factor is strictly positive.  Thus local integrability holds exactly
when it holds for every factor.  Taking the suprema of these downward
closed exponent sets proves the first identity in
\eqref{eq:product-arnold}, including zero or infinite thresholds; taking
reciprocals proves the second.
\end{proof}

\begin{remark}[Analytic singularities under products]
Pullback by each holomorphic projection $\pr_i$ preserves analytic
singularities.  More explicitly, suppose locally
$u_i=\frac{c_i}{2}\log\sum_\alpha|f_{i,\alpha}|^2+h_i$, where $h_i$ is smooth, and
suppose that the positive weights $c_i$ are commensurable.  Choose $c>0$
and integers $N_i\ge1$ with $c_i=N_ic$.  Then
\[
 \sum_i\pr_i^*u_i
 =\frac c2\log\prod_i\left(\sum_\alpha|f_{i,\alpha}|^2\right)^{N_i}
   +\sum_i\pr_i^*h_i,
\]
where $\log 0=-\infty$: at a pole of any factor, both sides are
$-\infty$.  Expansion of the product expresses the argument of the logarithm as a
finite sum of squared absolute values of holomorphic monomials.  Hence the
sum has analytic singularities.  Without commensurability, the sum need not
have neat analytic singularities.  This issue does not arise in \cref{cor:all-dimensions}, where only one
factor is singular. The later \cref{cor:profile-envelopes} allows several
singular factors and does not assert analytic singularities for arbitrary
real weights.
\end{remark}

\begin{corollary}
\label[corollary]{cor:all-dimensions}
For every pair of integers $n\ge2$ and $e\ge2$, there exist a smooth
projective $n$-fold $X_n$, a K\"ahler form $\omega_n$, a smooth positive
volume form $dV_n$, an initial potential
$\Phi_0\in\PSH(X_n,\omega_n)$ with analytic singularities, and a point
$x_n\in X_n$ with $\arn(\Phi_0,x_n)=1/2$, such that the maximal weak
solution satisfies
\begin{equation}\label{eq:all-dimensional-counterexample}
 \arn(\Phi_t,x_n)=\frac12-\frac te
 >\max\{\arn(\Phi_0,x_n)-t,0\},\qquad 0<t<e/2.
\end{equation}
\end{corollary}

\begin{proof}
For $n=2$, use \cref{cor:exact-counterexample}.  Suppose $n>2$ and apply
\cref{thm:product} to that surface flow and a smooth projective manifold
$Y$ of dimension $m:=n-2$ with K\"ahler form $\eta$.  Set
$dV_Y:=\eta^m$ and take the initial potential on $Y$ to be zero.  Since
$(\eta+\ddc0)^m=\exp(0)dV_Y$, this factor is stationary.  On
$X_n:=\mathbb F_e\times Y$, let $\pr_1$ and $\pr_2$ be its projections,
and take $\omega_n:=\pr_1^*\omega+\pr_2^*\eta$ and
\[
 dV_n:=\frac{n!}{2!\,(n-2)!}\,
 \pr_1^*dV_X\wedge\pr_2^*dV_Y.
\]
This is the normalized volume form in \eqref{eq:product-volume}.  Set
$\Phi_0:=\pr_1^*\varphi_0$, choose $x\in S\setminus F_1$ and $y\in Y$,
and put $x_n:=(x,y)$.
The potential $\Phi_0$ has analytic singularities, and
\eqref{eq:product-arnold} gives
$\arn(\Phi_t,x_n)=\arn(\varphi_t,x)=1/2-t/e$ for $0<t<e/2$.  The initial
value is $\arn(\Phi_0,x_n)=1/2$, which proves
\eqref{eq:all-dimensional-counterexample}.
\end{proof}

\begin{corollary}
\label[corollary]{cor:profile-envelopes}
Let $r\ge1$.  For $1\le i\le r$, choose an integer $e_i\ge2$, an integer
$k_i\ge1$, positive numbers $a_i,b_{i,1},\ldots,b_{i,k_i}$ satisfying
$\sum_{j=1}^{k_i}b_{i,j}>e_i a_i$, and a flow on
$X_i:=\mathbb F_{e_i}$ as in \cref{thm:Hirz-exact}.  Let $S_i$ be the
negative section, let $F_{i,1},\ldots,F_{i,k_i}$ be the selected fibres,
and choose
$x_i\in S_i\setminus\bigcup_{j=1}^{k_i}F_{i,j}$.  Define
\[
 \kappa_i:=\min\{k_i/e_i,1\},
 \qquad \tau_i:=\tfrac12\min\{a_i/\kappa_i,b_{i,1},\ldots,b_{i,k_i}\}.
\]
Let $(\Phi_t)$ be the product flow in \cref{thm:product} and set
$\mathbf x:=(x_1,\ldots,x_r)$.  For $0<t<\min_i\tau_i$,
\begin{equation}\label{eq:profile-envelope}
 \arn(\Phi_t,\mathbf x)=
 \max_{1\le i\le r}\{a_i/2-\kappa_it\}.
\end{equation}
\end{corollary}

\begin{proof}
Apply \eqref{eq:product-arnold} and \cref{thm:Hirz-exact} factor by
factor.
\end{proof}

\begin{remark}[A corner profile]
Take $(e_1,k_1,a_1)=(2,2,3)$ with $b_{1,1}=b_{1,2}=4$, and take
$(e_2,k_2,a_2)=(3,1,2)$ with $b_{2,1}=7$.  The K\"ahler inequalities are
$8>2\cdot3$ and $7>3\cdot2$, while $\tau_1=3/2$ and $\tau_2=3$.
The two factors have respective Arnold-multiplicity profiles
$3/2-t$ and $1-t/3$ on the common interval $0<t<3/2$.
The product flow has their maximum as its profile: it equals the first
branch for $0<t\le3/4$ and the second for $3/4\le t<3/2$.
More generally, \cref{cor:profile-envelopes} realizes finite maxima of
these affine profiles.  Such a maximum is convex and piecewise affine:
after identical branches are identified, distinct affine functions have
at most one intersection each.
It is strictly decreasing: for $s<t$ in the common interval and
$\kappa_*:=\min_i\kappa_i>0$, its value at $t$ is at most its value at
$s$ minus $\kappa_*(t-s)$.
\end{remark}

\appendix
\section{Resolution bounds and Zariski chambers}
\label{app:resolution-tools}

The resolution bounds below apply to DGL-compatible analytic initial
potentials, as defined next.  These may have non-pure-divisorial or
non-SNC singularities on $X$, but their coefficients satisfy an additional
restriction not imposed in \cref{thm:surface-exactness}.
We finish with the classical surface chamber calculation.

\subsection{Log resolutions and divisor coefficients}
\label{subsec:resolutions}

The following coefficient convention is used only in this appendix.
We call $u$ \emph{DGL-compatible} if there is a single positive integer
$m_0$ such that the local representations
\eqref{eq:analytic-singularities} can all be chosen with $c_0=1/m_0$.
On compact $X$, this is equivalent to allowing a positive integer
$m_U$ on each coordinate neighbourhood $U$, as in
\cite{DGL26geometry}*{Definition~1.1}.  Indeed, choose a finite such
cover, let $m_0$ be the least common multiple of its denominators, and
put $p_U=m_0/m_U$.  Replace each local tuple $(f_1,\ldots,f_N)$ by
all ordered products of $p_U$ of its entries.  These generate the
$p_U$-th power of the original ideal, and their squared norms sum
exactly to $(\sum_\ell|f_\ell|^2)^{p_U}$.  Thus the coefficient becomes
$1/(2m_0)$ while the smooth remainder $h$ is unchanged.

For this subsection, fix a K\"ahler form $\omega$ on $X$ and a potential
$\varphi_0\in\PSH(X,\omega)$ with DGL-compatible analytic singularities.
Fix one common denominator $m_0$ as above and put $c_0=1/m_0$.
On each coordinate neighbourhood $U$, retain an exact representation
\eqref{eq:analytic-singularities} with this coefficient and smooth
remainder, and let $\mathfrak a_U=(f_1,\ldots,f_N)$ be the ideal
generated by its holomorphic functions.  Here and below these are the
original generators, or the ordered-product generators used to obtain
the common denominator.

On overlaps, the two generator norms are locally comparable above and
below, since the coefficient is common and the smooth remainders differ
by a locally bounded function.  The analytic criterion for integral
closure therefore gives
$\overline{\mathfrak a_U}=\overline{\mathfrak a_V}$ on $U\cap V$:
a holomorphic germ $g$ belongs to $\overline{\mathfrak a_U}$ exactly
when $|g|\le C(\sum_\ell|f_\ell|^2)^{1/2}$ locally for some $C>0$.
Integral closures of coherent analytic ideals are coherent, so these
local closures glue to a coherent integrally closed ideal sheaf
$\mathfrak a\subset\mathcal O_X$; see
\cite{Dem-analytic}*{\S1.B, Definition~1.10 and Exercise~1.11}.
The real-weight SNC statements in the main text use their given divisors directly
and do not require this construction.
An
\emph{admissible principalizing log resolution} is a proper bimeromorphic
holomorphic map $\pi:Y\to X$, with $Y$ a smooth compact K\"ahler manifold,
such that $\pi$ is an
isomorphism over $X\setminus V(\mathfrak a)$, where $V(\mathfrak a)$ is
the zero set of $\mathfrak a$,
$\mathfrak a\mathcal O_Y$ is locally generated by one holomorphic
function, and the union of its divisor with the relative Jacobian divisor
has SNC support.  The relative Jacobian divisor is the zero divisor of
the local Jacobian determinant.  Such resolutions can be
chosen as compositions of blow-ups along smooth centres; the K\"ahler
property is preserved under these blow-ups.  Let
$D_1,\ldots,D_{q_D}$ be the distinct prime components of the divisor of
$\mathfrak a\mathcal O_Y$; thus $q_D\ge1$ in the nonsmooth case considered
below.  We require, as part of admissibility, that every prime component of
the relative Jacobian divisor occur among the $D_j$.  On such a resolution,
\begin{equation}\label{eq:resolution-data}
 \pi^*(\omega+\ddc\varphi_0)
 =\sum_{j=1}^{q_D}m_j[D_j]+R_{\pi,0},
 \qquad m_j>0,
\end{equation}
where $R_{\pi,0}$ is a smooth closed semipositive real $(1,1)$-form on
$Y$.  To justify both properties, work above one of the coordinate
neighbourhoods $U$ just chosen and on a chart where
$\mathfrak a\mathcal O_Y=\mathcal O_Y(-E)$, with
$E=\sum_ja_jD_j$ and local defining function $\sigma_E$.
The extended ideal $I:=\mathfrak a_U\mathcal O_Y$ is contained in
$(\sigma_E)$, so write $I=\sigma_E\mathfrak j$ for an ideal
$\mathfrak j\subset\mathcal O_Y$ on this chart.
Since $\mathfrak a|_U=\overline{\mathfrak a_U}$, integral dependence
is preserved under pullback and $(\sigma_E)$ is integral over $I$.
Cancelling powers of $\sigma_E$ in an integral equation for
$\sigma_E$ gives $1\in\overline{\mathfrak j}$, hence
$\overline{\mathfrak j}=\mathcal O_Y$.
This forces $\mathfrak j=\mathcal O_Y$: a monic integral equation
for $1$ has all its remaining coefficients in $\mathfrak j$, so
$1\in\mathfrak j$.
Thus $\mathfrak a_U\mathcal O_Y=(\sigma_E)$ as well.

Using the generators retained in the exact local expression for
$\varphi_0$, and shrinking to one common neighbourhood of any given
point, we therefore have
$\pi^*f_\ell=\sigma_Eu_\ell$ for every $\ell$, where the holomorphic
functions $u_\ell$ have no common zero on that neighbourhood.  Put
\[
 r:=\frac{c_0}{2}\log\!\left(\sum_\ell|u_\ell|^2\right)+h\circ\pi.
\]
The sum inside the logarithm is strictly positive, so $r$ is $C^\infty$;
no real-analytic regularity of the given smooth term $h$ is required.
On the complement of $E$ we have
\[
 \varphi_0\circ\pi-\frac{c_0}{2}\log|\sigma_E|^2=r.
\]
It is the right-hand side that extends smoothly across $E$; the displayed
subtraction is only taken off $E$.  Since
$\ddc\log|\sigma_E|^2=0$ there, the smooth form
$\pi^*\omega+\ddc r$ agrees with the pulled-back current off $E$.
The distributional Poincar\'e--Lelong identity then gives
\eqref{eq:resolution-data}, with $m_j=c_0a_j$ in the normalization
\eqref{eq:ddc-normalization}, and identifies its residual current with
this smooth form.  Positivity of the pulled-back current implies
semipositivity of that form off $E$; continuity of its coefficients and
density of the complement extend semipositivity across $E$.
Equivalently, Siu's decomposition removes the prime-divisor components
with their generic Lelong coefficients and leaves a positive current;
the Lelong-number convention is recalled in
\cref{subsec:arnold-invariants}.  Thus $R_{\pi,0}$ is smooth and
semipositive, not necessarily strictly positive.
Let $s_j$ be the canonical defining section of $\mathcal O_Y(D_j)$ and
choose a smooth Hermitian metric $h_j$ on that line bundle.  Let
$\theta_j$ be its smooth normalized curvature form from \eqref{eq:PL}, so
$\{\theta_j\}=\{D_j\}$.  By multiplying one of the metrics by a smooth
positive function, we may arrange the identities
\begin{equation}\label{eq:resolution-metric-normalization}
 \pi^*\omega=R_{\pi,0}+\sum_{j=1}^{q_D}m_j\theta_j,
 \qquad
 \varphi_0\circ\pi=\sum_{j=1}^{q_D}m_j\ell_j+C_{\pi,0},
 \qquad \ell_j:=\log|s_j|_{h_j},
\end{equation}
where $C_{\pi,0}$ is a real constant.  To obtain
\eqref{eq:resolution-metric-normalization}, begin with arbitrary metrics
$h_j$ and define on
$Y\setminus\bigcup_jD_j$
\[
 g:=\varphi_0\circ\pi-\sum_{j=1}^{q_D}m_j\ell_j
\]
and use the calculation above to extend $g$ smoothly to $Y$.  Equation
\eqref{eq:resolution-data} and the Poincar\'e--Lelong identity give
\[
 \pi^*\omega=R_{\pi,0}+\sum_{j=1}^{q_D}m_j\theta_j-\ddc g.
\]
Replace $h_1$ by $e^{2(g-C_{\pi,0})/m_1}h_1$.  This adds
$(g-C_{\pi,0})/m_1$ to $\ell_1$ and subtracts $\ddc g/m_1$ from
$\theta_1$, giving \eqref{eq:resolution-metric-normalization}.
Let
\[
 K_{Y/X}:=\operatorname{div}(\det d\pi)
 =\sum_{j=1}^{q_D}b_jD_j,\qquad b_j\in\mathbb Z_{\ge0},
\]
be the relative canonical divisor, defined by the Jacobian determinant
section.  A smooth positive volume form $dV_Y$ on $Y$ can then be chosen
so that
\begin{equation}\label{eq:jacobian}
 \pi^*dV_X=\left(\prod_{j=1}^{q_D}|s_j|_{h_j}^{2b_j}\right)dV_Y.
\end{equation}
Here $\operatorname{ord}_{D_j}$ denotes order of vanishing along $D_j$;
on a divisor, it extracts the coefficient of $D_j$.
Thus $b_j$ is the discrepancy coefficient and the log discrepancy of the
divisorial valuation $\operatorname{ord}_{D_j}$ is
\begin{equation}\label{eq:log-discrepancy}
 A_X(D_j)=1+b_j.
\end{equation}
The centre of $D_j$ on $X$ is $c_X(D_j)=\pi(D_j)$.

We keep the differential-operator convention
\eqref{eq:ddc-normalization} and the exponential-integrability convention
of \cite{DGL26geometry}*{Definition~1.2}.  With the canonical integration
current fixed, a divisor coefficient and a squared-logarithm coefficient
must be distinguished:
\begin{equation}\label{eq:DGL-normalization-conversion}
 m_j\ell_j=\frac{m_j}{2}\log|s_j|_{h_j}^2,
 \qquad
 \ddc\log|s_j|_{h_j}^2=2[D_j]-2\theta_j.
\end{equation}
Thus the logarithmic coefficients in the potential-level barrier argument
of \cite{DGL26geometry}*{Lemma~1.7} are $k_j=m_j/2$.
We apply that argument at the same time $t$, without identifying its
squared-logarithm coefficients with canonical current coefficients.
For completeness, put $A_j=1+b_j$ and
$\Omega_Y:=Y\setminus\bigcup_jD_j$.  On $\Omega_Y$ consider the
partially regularized upper barrier
\[
 v_{t,\varepsilon}
 =\sum_j\left((k_j-t)\log(|s_j|_{h_j}^2+\varepsilon^2)
              -b_jt\log|s_j|_{h_j}^2\right)+C_{\pi,0}.
\]
Only the first logarithms are regularized; the discrepancy terms
$-b_jt\log|s_j|_{h_j}^2$ remain singular along the Jacobian divisor.
For $0<t<\min_jm_j/(2A_j)$, all $k_j-t$ are positive.  The rank-one
estimate in the cited proof gives, on $\Omega_Y$ at points where
$\pi^*\omega+\ddc v_{t,\varepsilon}\ge0$,
\[
 (\pi^*\omega+\ddc v_{t,\varepsilon})^n
 \le C\prod_j(|s_j|_{h_j}^2+\varepsilon^2)^{-1}dV_Y
 =C\exp(\partial_tv_{t,\varepsilon})\pi^*dV_X.
\]
The last equality uses \eqref{eq:jacobian}.  For $C\ge1$,
$v_{t,\varepsilon}+t\log C$ satisfies the supersolution inequality
on $\Omega_Y$ in the sense of
\cite{DGL26geometry}*{Definition~1.4}, and its initial value there
dominates $\varphi_0\circ\pi$.
The comparison argument of \cite{DGL26geometry}*{Lemma~1.7},
followed by $\varepsilon\downarrow0$, yields
\[
 \varphi_t\circ\pi
 \le\sum_j(m_j-2A_jt)\ell_j+C_{\pi,0}+Ct.
\]
Our Poincar\'e--Lelong identity now implies that the generic divisorial
Lelong number of $\pi^*(\omega+\ddc\varphi_t)$ along $D_j$ is at least
$m_j-2A_jt$.  Siu decomposition leaves a positive current after these
prescribed terms are removed.  Using
\eqref{eq:resolution-metric-normalization}, we obtain
\begin{equation}\label{eq:DGL-potential-conversion}
 \varphi_t\circ\pi
 =\sum_j(m_j-2A_jt)\ell_j+u_t,
 \qquad
 u_t\in\PSH\!\left(Y,R_{\pi,0}+2t\sum_jA_j\theta_j\right).
\end{equation}
Here $u_t$ is first defined off the divisor and then extended
quasi-plurisubharmonically by the positive residual current.  This
formulation keeps all divisor coefficients in the canonical current
normalization.

\subsection{Dynamic divisorial lower bounds}
\label{sec:dynamic}

The next estimate combines divisorial extraction on an admissible
resolution with domination by the Zariski negative part.

\begin{theorem}\label[theorem]{thm:dynamic}
Let $(X,\omega)$ be a compact K\"ahler manifold, let $dV_X$ be a smooth
positive volume form, and let $\varphi_0\in\PSH(X,\omega)$ have
DGL-compatible analytic singularities.  Let
$\pi:Y\to X$ be an admissible principalizing log resolution of the
ideal $\mathfrak a$ obtained from a fixed common denominator as in
\cref{subsec:resolutions}, with data
\eqref{eq:resolution-data}--\eqref{eq:jacobian}, where $q_D\ge1$.  Denote
by $(\varphi_t)$ the maximal weak solution from $\varphi_0$ and put
\begin{equation}\label{eq:dynamic-data}
 A_j=1+b_j,\qquad
 t_1=\min_j\frac{m_j}{2A_j},\qquad
 \theta_\pi=\sum_{j=1}^{q_D}A_j\theta_j,\qquad
 \alpha_t=\{R_{\pi,0}+2t\theta_\pi\}
 \in H^{1,1}(Y,\mathbb R).
\end{equation}
For $0<t<t_1$ there is a closed positive $(1,1)$-current $Q_t$ on $Y$
such that
\begin{equation}\label{eq:extracted-current}
 \pi^*(\omega+\ddc\varphi_t)
 =\sum_{j=1}^{q_D}(m_j-2A_jt)[D_j]+Q_t,
 \qquad \{Q_t\}=\alpha_t.
\end{equation}
In particular $\alpha_t$ is pseudoeffective and
\begin{equation}\label{eq:dynamic-current-bound}
 \pi^*(\omega+\ddc\varphi_t)
 \ge \sum_{j=1}^{q_D}(m_j-2A_jt)[D_j]+[N(\alpha_t)].
\end{equation}
For every $x\in X$,
\begin{equation}\label{eq:valuation-bound}
 \arn(\varphi_t,x)\ge
 \max_{\substack{1\le j\le q_D\\x\in c_X(D_j)}}
 \frac{m_j-2A_jt+\nu_{D_j}(\alpha_t)}{2A_j},
\end{equation}
where the maximum over an empty set is zero.
\end{theorem}

\begin{proof}
Set $T_t^Y:=\pi^*T_t=\pi^*(\omega+\ddc\varphi_t)$.  By
\eqref{eq:DGL-potential-conversion}, there is a function $u_t$ such that
\[
 \varphi_t\circ\pi
 =\sum_{j=1}^{q_D}(m_j-2A_jt)\ell_j+u_t,
 \qquad
 u_t\in\PSH(Y,R_{\pi,0}+2t\theta_\pi).
\]
Applying the Poincar\'e--Lelong identity \eqref{eq:PL} term by term and
using \eqref{eq:resolution-metric-normalization} gives
\[
 T_t^Y=\sum_{j=1}^{q_D}(m_j-2A_jt)[D_j]+Q_t,
 \qquad
 Q_t:=R_{\pi,0}+2t\theta_\pi+\ddc u_t\ge0.
\]
Taking cohomology and using \eqref{eq:resolution-data} and
$\{D_j\}=\{\theta_j\}$ gives
\[
 \{Q_t\}
 =\pi^*\{\omega\}-\sum_{j=1}^{q_D}(m_j-2A_jt)\{D_j\}
 =\{R_{\pi,0}\}+2t\sum_{j=1}^{q_D}A_j\{D_j\}
 =\{R_{\pi,0}+2t\theta_\pi\}=\alpha_t.
\]
This proves \eqref{eq:extracted-current} and shows, before any use of a
Zariski decomposition, that $\alpha_t$ is pseudoeffective.  Applying
\eqref{eq:Bou-domination} to the positive current $Q_t$ yields
\eqref{eq:dynamic-current-bound}.  The current $Q_t$ may still have
additional divisorial Siu components; thus \eqref{eq:extracted-current} is
an extraction of the prescribed components, not the complete Siu
decomposition.  The form $R_{\pi,0}+2t\theta_\pi$ is merely a smooth
representative of $\alpha_t$; it need not be semipositive.

Fix $j$ with $x\in c_X(D_j)$ and set
\[
 \gamma_j(t):=m_j-2A_jt+\nu_{D_j}(\alpha_t)>0.
\]
Let $U$ be an arbitrary coordinate neighbourhood of $x$.  Choose a smooth
point $y\in D_j\cap\pi^{-1}(U)$ that lies on none of the other $D_\ell$.
Such points form a dense open subset of $D_j\cap\pi^{-1}(U)$.  By
\eqref{eq:dynamic-current-bound}, $T_t^Y-\gamma_j(t)[D_j]$ is positive.  In
coordinates $(z_1,\ldots,z_n)$ centred at $y$ with $D_j=\{z_1=0\}$, a
local potential therefore has the form
\[
 \varphi_t\circ\pi=\gamma_j(t)\log|z_1|+w+h,
\]
where $w$ is plurisubharmonic and $h$ is smooth.  On a relatively compact
polydisc about $y$, $w+h$ is bounded above, while
\eqref{eq:jacobian} takes the form
\[
 \pi^*dV_X=|z_1|^{2b_j}g\,dV_{\rm eucl},
 \qquad g>0\ \text{smooth}.
\]
Here $dV_{\rm eucl}$ is the Euclidean volume form in these coordinates.
Tonelli's theorem and change of variables therefore imply that, for every
$q>0$, there are constants $C,\varepsilon>0$ such that
\[
 \int_U\exp(-q\varphi_t)dV_X
 \ge C\int_{|z_1|<\varepsilon}
 |z_1|^{2b_j-q\gamma_j(t)}dA(z_1)
\]
where $dA$ is Euclidean area measure.  The last integral diverges
exactly when $q\gamma_j(t)\ge2(1+b_j)=2A_j$.  Since $U$ was arbitrary,
$\lct(\varphi_t,x)\le2A_j/\gamma_j(t)$, and hence
$\arn(\varphi_t,x)\ge\gamma_j(t)/(2A_j)$.  Maximizing over all such $j$
proves \eqref{eq:valuation-bound}; if no $D_j$ has centre containing $x$,
the stated bound is zero by convention.
\end{proof}

\begin{remark}[Surface chambers]\label[remark]{prop:chamber}
Let $X$ be a compact K\"ahler surface and let
$\alpha_t=\alpha_0+t\gamma$ be big on an interval $I$.
On a subinterval where $N(\alpha_t)=\sum_{i=1}^r n_i(t)C_i$ has fixed
nonempty support, the matrix $M=(C_i\cdot C_j)_{i,j}$ is negative
definite.  Orthogonality $P(\alpha_t)\cdot C_j=0$ gives
\begin{equation}\label{eq:chamber-matrix}
 (n_i(t))_{i=1}^r=M^{-1}(\alpha_t\cdot C_j)_{j=1}^r.
\end{equation}
Thus the coefficients are affine there; the empty-support case is
immediate.  If $X$ is projective and $\alpha_0,\gamma$ belong to the
real N\'eron--Severi space $N^1(X)_{\mathbb R}$ of divisor classes
modulo numerical equivalence, local finiteness of Zariski chambers
\cite{BKS04}*{Theorem~1.1 and Proposition~1.13} shows that only finitely
many curves occur and their coefficients are piecewise affine on each
compact subinterval $J\Subset I$.
\end{remark}

\section*{AI Use Declaration}

The main idea of this work arose from exploratory mathematical discussions
with ChatGPT (OpenAI), using the GPT-5.6 Sol model.  ChatGPT also assisted in preparing and refining the Lean formalization code. The author takes full responsibility for the paper and its accompanying formalization materials, including the mathematical arguments, references, formal statements, and their correspondence.

\section*{Computer-assisted verification}

The accompanying Lean~4 supplement records the deductions and explicit
standard inputs covered by its documentation, including the SNC
multiplier-ideal computation and positive-slope wall statements.  It does
not reconstruct the underlying pluripotential, geometric, or weak-flow
theories from first principles.  The bound $\eta_i\le1-m_i/M$, the
automatic auxiliary-divisor construction, the pointwise and full-mass
conclusions for the actual residual current, and the nonlocality
construction for arbitrary $a>0$ have not been verified in Lean.
The extension to arbitrary dimension and the interval $(0,\tau_Z)$
have not been checked in Lean either.  This
revision has not been rerun or independently checked against the formal
sources.  The sources and precise formal scope are documented in the
\href{https://github.com/qinxiangsen/failure-of-local-arnold-multiplicity-decay}{accompanying repository}.

\begingroup
\makeatletter
\renewcommand{\BibLabel}{%
  \Hy@raisedlink{\hyper@anchorstart{cite.\CurrentBib}\hyper@anchorend}%
  [\thebib]%
}
\makeatother\ \\ \

\endgroup


\begin{bibdiv}
\begin{biblist}

\bib{BKS04}{article}{
 label={BKS04},
 author={Bauer, Thomas},
 author={K{\"u}ronya, Alex},
 author={Szemberg, Tomasz},
 title={Zariski chambers, volumes, and stable base loci},
 journal={J. Reine Angew. Math.},
 volume={576},
 date={2004},
 pages={209--233},
}

\bib{BK07}{article}{
 label={BK07},
 author={B{\l}ocki, Zbigniew},
 author={Ko{\l}odziej, S{\l}awomir},
 title={On regularization of plurisubharmonic functions on manifolds},
 journal={Proc. Amer. Math. Soc.},
 volume={135},
 number={7},
 date={2007},
 pages={2089--2093},
 doi={10.1090/S0002-9939-07-08858-2},
}

\bib{BG14}{article}{
 label={BG14},
 author={Berman, Robert J.},
 author={Guenancia, Henri},
 title={K{\"a}hler--Einstein metrics on stable varieties and log canonical pairs},
 journal={Geom. Funct. Anal.},
 volume={24},
 number={6},
 date={2014},
 pages={1683--1730},
}

\bib{Bou04}{article}{
 label={Bou04},
 author={Boucksom, S{\'e}bastien},
 title={Divisorial Zariski decompositions on compact complex manifolds},
 journal={Ann. Sci. {\'E}c. Norm. Sup{\'e}r. (4)},
 volume={37},
 number={1},
 date={2004},
 pages={45--76},
}

\bib{BEGZ10}{article}{
 label={BEGZ10},
 author={Boucksom, S{\'e}bastien},
 author={Eyssidieux, Philippe},
 author={Guedj, Vincent},
 author={Zeriahi, Ahmed},
 title={Monge--Amp{\`e}re equations in big cohomology classes},
 journal={Acta Math.},
 volume={205},
 number={2},
 date={2010},
 pages={199--262},
}

\bib{DDL18}{article}{
 label={DDL18},
 author={Darvas, Tam{\'a}s},
 author={Di Nezza, Eleonora},
 author={Lu, Chinh H.},
 title={On the singularity type of full mass currents in big cohomology classes},
 journal={Compos. Math.},
 volume={154},
 number={2},
 date={2018},
 pages={380--409},
}

\bib{Dem-analytic}{misc}{
 label={Dem09},
 author={Demailly, Jean-Pierre},
 title={Analytic methods in algebraic geometry},
 date={2009},
 note={Lecture notes, December 2009. Available from the
 \href{https://www-fourier.univ-grenoble-alpes.fr/~demailly/manuscripts/analmeth.pdf}{author's website}},
}

\bib{Dem-book}{misc}{
 label={Dem12},
 author={Demailly, Jean-Pierre},
 title={Complex analytic and differential geometry},
 date={2012},
 note={Online monograph, 2012 version. Available from the
 \href{https://www-fourier.univ-grenoble-alpes.fr/~demailly/manuscripts/agbook.pdf}{author's website}},
}

\bib{DFT17}{article}{
 label={DFT17},
 author={Di Nezza, Eleonora},
 author={Floris, Enrica},
 author={Trapani, Stefano},
 title={Divisorial Zariski decomposition and some properties of full mass currents},
 journal={Ann. Sc. Norm. Super. Pisa Cl. Sci. (5)},
 volume={17},
 number={4},
 date={2017},
 pages={1383--1396},
}

\bib{DGL26geometry}{misc}{
 label={DGL26a},
 author={Di Nezza, Eleonora},
 author={Guedj, Vincent},
 author={Lu, Chinh H.},
 title={Geometric smoothing by the K{\"a}hler--Ricci flow},
 date={2026},
 note={\href{https://arxiv.org/abs/2607.08325v1}{arXiv:2607.08325v1}},
}

\bib{DGL26cusps}{misc}{
 label={DGL26b},
 author={Di Nezza, Eleonora},
 author={Guedj, Vincent},
 author={Lu, Chinh H.},
 title={K{\"a}hler--Ricci flow: from divisors to cusps},
 date={2026},
 note={\href{https://arxiv.org/abs/2606.15816v1}{arXiv:2606.15816v1}},
}

\bib{DNL17}{article}{
 label={DNL17},
 author={Di Nezza, Eleonora},
 author={Lu, Chinh H.},
 title={Uniqueness and short time regularity of the weak K{\"a}hler--Ricci flow},
 journal={Adv. Math.},
 volume={305},
 date={2017},
 pages={953--993},
}

\bib{GZ17book}{book}{
 label={GZ17a},
 author={Guedj, Vincent},
 author={Zeriahi, Ahmed},
 title={Degenerate complex Monge--Amp{\`e}re equations},
 series={EMS Tracts in Mathematics},
 volume={26},
 publisher={European Mathematical Society (EMS)},
 address={Z{\"u}rich},
 date={2017},
}

\bib{GZ17flow}{article}{
 label={GZ17b},
 author={Guedj, Vincent},
 author={Zeriahi, Ahmed},
 title={Regularizing properties of the twisted K{\"a}hler--Ricci flow},
 journal={J. Reine Angew. Math.},
 volume={729},
 date={2017},
 pages={275--304},
}

\bib{Hartshorne}{book}{
 label={Har77},
 author={Hartshorne, Robin},
 title={Algebraic Geometry},
 series={Graduate Texts in Mathematics},
 volume={52},
 publisher={Springer},
 address={New York},
 date={1977},
}

\bib{Siu74}{article}{
 label={Siu74},
 author={Siu, Yum-Tong},
 title={Analyticity of sets associated to Lelong numbers and the extension of closed positive currents},
 journal={Invent. Math.},
 volume={27},
 date={1974},
 pages={53--156},
}

\end{biblist}
\end{bibdiv}
\end{document}